\documentclass[11pt]{article}
\usepackage{mymacros}
\usepackage{comment}
\usepackage{mathtools}
\usepackage{color}

\usepackage{soul}

\newcommand{\EE}{\bm{E}}
\newcommand{\genmin}{u}		
\newcommand{\shgmin}{\bar u} 
\newcommand{\lvshgmin}{\bar u} 
\newcommand{\mgenmin}{v^\epsilon} 

\newcommand{\ShG}{\mathrm{ShG}}
\newcommand{\Lv}{\mathrm{Lv}}
\newcommand{\GFF}{\textup{GFF}}

\newcommand{\PhiGFF}{\Phi^\epsilon}
\newcommand{\PhitShG}{\Psi^\epsilon}

\newcommand{\MGFFeps}{M^{\GFF_\epsilon}}

\newcommand{\MpmShGeps}{M^{\ShG_\epsilon,t,\pm}}

\title{A coupling for the Liouville and the sinh-Gordon model in the $L^2$ phase
}

\author{Michael Hofstetter \footnote{Department of Mathematics,
Weizmann Institute.  E-mail: {\tt michael.hofstetter@weizmann.ac.il}}
}

\date{}

\begin{document}

\maketitle

\begin{abstract}
Using a stochastic control approach we establish couplings of the Liouville field and the sinh-Gordon field with the Gaussian free field in dimension $d=2$, such that the difference is in a Sobolev space of regularity $\alpha>1$.
The analysis covers the entire $L^2$ phase.
Our main tools are estimates for the short scales of the minimiser of the variational problem and several applications of the Brascamp-Lieb inequality.
\end{abstract}




\section{Introduction}
\subsection{Model and main results}

In this work we study the Liouville model and the sinh-Gordon model on the unit torus $\Omega= \T^2$ in dimension $d=2$ in the $L^2$ phase.
These models are probability measure on $S'(\Omega)$ formally defined as
\begin{equation}
\label{eq:exp-model-formal-density}
d\nu^{\cE}(\phi) \propto \exp\Big[{-\lambda \int_{\Omega} V(\phi_x)dx} \Big] d\nu_m^{\GFF} (\phi) ,
\end{equation}
where $\nu_m^{\GFF}$ is the Gaussian free field on $\Omega$ with mass $m>0$,
i.e., the centred Gaussian measure on $\Omega$ with covariance $(-\Delta +m^2)^{-1}$, $\lambda>0$ and $V(\phi_x) = \exp(\sqrt{\beta}(\phi_x))$ for the Liouville model and $V(\phi_x)= \cosh(\sqrt{\beta}(\phi_x))$ for the sinh-Gordon model.
To distinguish the two models, we use the notation $\cE= \Lv$ and $\cE = \ShG$ to refer to the Liouville and the sinh-Gordon model.
The exponential interaction relates these models to the theory of Gaussian multiplicative chaos, for which the model is non-trivial when $\beta \in (0, 8\pi)$.
Throughout this work, we are mainly concerned with the case $\beta \in (0,4\pi)$, also known as the $L^2$ phase of the Gaussian multiplicative chaos.

To turn the ill-defined expression \eqref{eq:exp-model-formal-density} into a well-defined object,
we use a lattice regularisation of underlying space $\Omega$ given by $\Omega_\epsilon = \Omega\cap \epsilon \Z^2$,
replace the continuum and distribution valued Gaussian free field by the discrete Gaussian free field on $\Omega_\epsilon$
and the non-linearity $V$ by its Wick ordering.
In the present context, this leads to regularised measures with densities on $X_\epsilon = \R^{\Omega_\epsilon}$ that are well-defined for any $\epsilon>0$. 
For the Liouville model we have
\begin{equation}
\label{eq:lv-density}
\nu^{\Lv_\epsilon}(d\phi)
\propto
\exp{\Big[-\lambda \int_{\Omega_\epsilon} \wick{\exp(\sqrt{\beta} \phi_x)}_\epsilon \Big]} d\nu_m^{\GFF_\epsilon} (\phi),
\end{equation}
while for the sinh-Gordon model, this recipe leads to
\begin{equation}
\label{eq:shg-density}
\nu^{\ShG_\epsilon}(d\phi)
\propto
\exp{\Big[-\lambda \int_{\Omega_\epsilon} \wick{\cosh(\sqrt{\beta} \phi_x)}_\epsilon \Big]} d\nu_m^{\GFF_\epsilon} (\phi),
\end{equation}
where the Wick ordering in each case is given by
\begin{equation}
\wick{\exp(\pm \sqrt{\beta} \phi_x)}_\epsilon = \epsilon^{\beta /4\pi} \exp(\pm\sqrt{\beta} \phi_x)
\end{equation}
and the discrete integral is defined by
\begin{equation}
\int_{\Omega_\epsilon} \varphi_x dx = \epsilon^2 \sum_{x\in \Omega_\epsilon} \varphi_x, \qquad \varphi \in X_\epsilon.
\end{equation}

The range of $\beta$ as well as the exponent of $\epsilon>0$ stem from the scaling of the reference measure $\nu_m^{\GFF_\epsilon}$ in \eqref{eq:lv-density} and \eqref{eq:shg-density}.
More precisely, we have, for $\phi\sim \nu_m^{\GFF_\epsilon}$,
\begin{equation}
\label{eq:scaling-gff}
\var(\phi_x) = \frac{1}{2\pi} \log \frac{1}{\epsilon} + O_m(1), \qquad \epsilon \to 0,
\end{equation}
where $O_m(1)$ denotes a function of $\epsilon$ and depending on $m$,
which remains bounded as $\epsilon \to 0$.

Our main focus is to study the relation of non-Gaussian distributions $\nu^{\cE_\epsilon}$ and the Gaussian free field $\nu^{\GFF_\epsilon}$ on all scales and uniformly in the lattice spacing $\epsilon>0$.
To this end, we define the Gaussian process
\begin{equation}
\label{eq:gff-process-eps}
\Phi_t^{\GFF_\epsilon} = \int_t^\infty \big(\dot c_s^\epsilon \big)^{1/2} dW_s^\epsilon , \qquad t\in [0,\infty),
\end{equation}
where $(\dot c_t^\epsilon)_{t\in [0,\infty]}$ is a continuous decomposition of the Gaussian free field covariance $c_\infty^\epsilon= (-\Delta^\epsilon + m^2)^{-1}$ and $W^\epsilon$ is a Brownian motion indexed by $\Omega_\epsilon$ and with quadratic variation $t/\epsilon^2$.
Here, $\Delta^\epsilon \colon X_\epsilon \to X_\epsilon$ is the lattice Laplacian acting on functions $f\in X_\epsilon$ by
\begin{equation}
\big(\Delta^\epsilon f\big) (x) = \epsilon^{-2} \sum_{y\sim x} \big( f(y) - f(x) \big),
\end{equation}
where $x\sim y$ denotes that $x,y \in \Omega_\epsilon$ are nearest neighbours in $\Omega_\epsilon$.
We note that $\Phi_0^{\GFF_\epsilon}$, i.e., the evaluation of the process $(\Phi_t^{\GFF_\epsilon})_{t\in [0,\infty]}$ at $t=0$ is a realisation of the Gaussian free field on $\Omega_\epsilon$.
In this article, we use the Pauli-Villars decomposition, which is given by
\begin{equation}
\label{eq:ct-pauli-villars}
c_t^\epsilon = (-\Delta^\epsilon + m^2 + 1/t)^{-1}.
\end{equation}
Note that, as $\epsilon \to 0$ and $x_\epsilon \in \Omega_\epsilon$,
\begin{equation}
\label{eq:ct-asymptotics}
c_t^\epsilon(x,x) \sim \frac{1}{2\pi} \log \frac{L_t}{\epsilon} + O_m(1),
\end{equation}
where $O_m(1)$ denotes a function which remains bounded as $\epsilon \to 0$, and where $L_t = \sqrt{t} \wedge 1/m$.

Our main result is a coupling between the Gaussian process $(\Phi_t^{\GFF_\epsilon})_{t\in [0,\infty]}$ and an analogous process $(\Phi_t^{\cE_\epsilon})_{t\in [0,\infty]}$ satisfying $\Phi_0^{\cE_\epsilon} \sim \nu^{\cE_\epsilon}$, such that their difference, below denoted as $(\Phi_t^{\Delta_\epsilon})_{t\in [0,\infty]}$,
is a bounded and regular random field with values in a Sobolev space $H^\alpha$ of regularity $\alpha>1$.
Since this statement holds uniformly in the lattice spacing, we are able to establish a continuum version of the coupling.
While the statements are identical for the Liouville model and the sinh-Gordon model,
the proof differs in these two cases at various steps.
We write $\cE$ when results hold for both Liouville and sinh-Gordon,
and we write $\Lv$ and $\ShG$ explicitely otherwise.
In the statement below, we denote by $C_0([0,\infty), S)$ the space of continuous processes on $[0,\infty)$ with values in a metric space $S$ vanishing at $\infty$.

\begin{theorem}
\label{thm:coupling-pphi-to-gff-eps}
Let $\beta \in (0, 4\pi)$.
For $\cE \in \{\Lv, \ShG\}$ and $\epsilon>0$,
there exists a process $\Phi^{\cE_\epsilon} \in C_0([0,\infty), H^{-\kappa})$ for any $\kappa >0$ such that
\begin{equation}
\label{eq:coupling-pphi-to-gff-eps}
\Phi_t^{\cE_\epsilon}
= \Phi_t^{\Delta_\epsilon}
+ \Phi_t^{\GFF_\epsilon}, \qquad  \Phi_0^{\cE_\epsilon} \sim \nu^{\cE_\epsilon} ,
\end{equation}
where the difference field $\Phi^{\Delta_\epsilon}$ satisfies, for any $t_0> 0$,
\begin{align}
\label{eq:phi-delta-h1}
 &\sup_{\epsilon >0 } \sup_{t\geq 0} \E [\|\Phi_t^{\Delta_\epsilon}\|_{H^{1}}^2] < \infty,
 \\
\label{eq:phi-delta-h2}
&\sup_{\epsilon>0} \sup_{t\geq t_0} \E [\|\Phi_t^{\Delta_\epsilon}\|_{H^{2}}^2] < \infty,
\end{align}
and moreover, for any $\alpha \in [1,2- \beta/4\pi)$
\begin{align} 
\label{eq:phi-delta-1-2}
&\sup_{\epsilon >0 } \sup_{t\geq 0} \E [\|\Phi_t^{\Delta_\epsilon}\|_{H^\alpha}] < \infty, \\
\label{eq:phi-delta-to-0}
&\sup_{\epsilon>0} \E[ \| \Phi_0^{\Delta_\epsilon}-\Phi_t^{\Delta_\epsilon}\|_{H^\alpha} ]\to 0 \text{~as~} t\to 0.
\end{align}  
Finally, for any $t>0$, $\Phi^{\GFF_\epsilon}_0-\Phi_t^{\GFF_\epsilon}$ is independent of $\Phi_t^{\cE_\epsilon}$.
\end{theorem}

\begin{corollary}
\label{cor:pphi-coupling-continuum}
Let $\beta\in (0,4\pi)$. There exists a process $\Phi^{\cE_0} \in C_0([0,\infty), H^{-\kappa})$ for every $\kappa>0$ such that, for any $t\geq 0$,
\begin{equation}
\label{eq:coupling-pphi-to-gff-cont}
\Phi_t^{\cE_0}
= \Phi_t^{\Delta_0}
+ \Phi_t^{\GFF_0},
\end{equation}
where $\Phi_0^{\cE_0}$ is distributed as the  continuum  Liouville measure (for $\cE = \Lv$) and sinh-Gordon measure (for $\cE= \ShG$),
and $\Phi_0^{\GFF_0}$ is distributed as the continuum Gaussian free field on $\Omega$.
For the difference field $\Phi^{\Delta_0}$,
the analogous estimates as for $\Phi_t^{\Delta_\epsilon}$ in Theorem \ref{thm:coupling-pphi-to-gff-eps} hold in the continuum Sobolev spaces. 
Finally, for any $t>0$, $\Phi_0^\GFF- \Phi_t^\GFF$ is independent of $\Phi_t^\cE$.
\end{corollary}

\begin{remark}
In dimension $d=2$ these results give a control on H\"older norms of suitable regularity $s\in (0,1)$ defined by
\begin{equation}
\| f \|_{C^s(\Omega)}  \equiv \| f \|_{C^s} = |f|_{C^s}  + \|f\|_{L^\infty}, \quad  |f|_{C^s(\Omega)}= \sup_{x,y \in \Omega, x\neq y} \frac{|f(x)-f(y)|}{|x-y|^s}.
\end{equation}
More precisely, for $s\in (0,1)$ and $\alpha \geq 1+s$, we have by standard Sobolev embeddings
\begin{equation}
\| f \|_{C^s(\Omega)} \lesssim_{\alpha, s} \|  f \|_{H^\alpha(\Omega)} .
\end{equation}
In particular, \eqref{eq:phi-delta-1-2} and \eqref{eq:phi-delta-to-0} give bounds on the regularity and the maximum of the difference field $\Phi_0^{\Delta_\epsilon}$,
which is uniform in the lattice width $\epsilon>0$.
The two main results therefore allow to study the extreme values of the non-Gaussian field $\Phi_0^\cE$ by comparision to the ones of the Gaussian free field.
Following the same proof as in \cite[Section 6]{MR4665719} for the $\cP(\phi)_2$ field,
this then implies that the sequence
\begin{equation}
\max_{\Omega_\epsilon} \Phi_0^{\cE_\epsilon} - m_\epsilon, \qquad m_\epsilon = \frac{1}{\sqrt{2\pi}}\big( 2\log \frac{1}{\epsilon} - \frac{3}{4} \log \log \frac{1}{\epsilon} \big)
\end{equation}
converges in law to a randomly shifted Gumbel distribution.

For $\cE= \Lv$, this result is known for $\beta\in (0,8\pi)$ thanks to \cite[Theorem 1.2]{HofstetterZeitouni2025Liouville},
while for $\cE = \ShG$ it extends the literature.
\end{remark}

\begin{remark}
We believe that the regularity of $\Phi^{\Delta_\epsilon}$ in Theorem \ref{thm:coupling-pphi-to-gff-eps} and Corollary \ref{cor:pphi-coupling-continuum} is not optimal,
and that similar bounds are true in the entire $L^1$ phase.
More precisely, we conjecture that for $\beta \in (0,8\pi)$ the difference field satisifes
\begin{equation}
\sup_{\epsilon>0} \E\big[ \| \Phi_0^{\Delta_\epsilon} \|_{H^\alpha} \big] <\infty
\end{equation}
for all $\alpha \in [1, \alpha^*)$, where $\alpha^*$ depends on $\beta$ and satisfies $\alpha^*>0$ for $\beta\in (0,8\pi)$ and $\alpha^* =  1$ for $\beta= 8\pi$.
For the heuristcs of this conjecture, we refer to Remark \ref{rem:identity-scmall-scale-drift}.
\end{remark}

\begin{remark}
In \cite[Theorem 1.1]{HofstetterZeitouni2025Liouville} it is shown that, for $\beta \in (0,8\pi)$,
the H\"older norms of $\Phi_0^{\Delta}$ of certain regularity depending on $\beta$ are bounded a.s.  
The regularity obtained in this work is quantified in expectation rather than a.s.,
and also the H\"older exponents differs. 
\end{remark}

\subsection{Literature}

The present work continues the analysis of Euclidean field theories in dimension $d=2$ through the lens of multiscale couplings to the Gaussian free field.
This program was initiated in \cite{MR4399156} for the sine-Gordon model, and subsequently continued in \cite{MR4665719} for the $\cP(\phi)_2$ fields and in \cite{HofstetterZeitouni2025Liouville} for the Liouville model in the entire $L^1$ phase.
The main contribution of this work is a coupling for the sinh-Gordon field,
which is not covered by existing results.

Our main tools are a stochastic control formulation, which allows to express the partition function of the measures in \eqref{eq:lv-density} and \eqref{eq:shg-density} as a variational formula,
as well as ideas from the equivalent Polchinski renormalisation group approach.
The interplay between these different points of view was first systematically used in \cite{MR4665719} for the $\cP(\phi)_2$ fields to obtain bounds for the minimiser of the variational problem,
which then imply bounds on the Sobolev norm of the difference field.
The present work follows in large parts the same strategy,
but the main input to obtain the estimates on the minimiser in the present work is different.

The Polchinski renormalisation group approach in the context of quantum field theory was rigorously developed in \cite{MR4303014} and successfully applied to prove a log-Sobolev inequality for the sine-Gordon measure for $\beta < 6\pi$.
The idea to apply a stochastic control approach to quantum field theory originates from \cite{MR4173157}, where this method was used to give a construction of the $\phi_d^4$ field in dimensions $d= 2$ and $d=3$.
For more details on the equivalence of these two approaches, we refer to the survey article \cite{MR4798104}.

Euclidean field theories with exponential interaction were introduced in \cite{MR292433},
and later studied in \cite{Albeverio1973Uniqueness} \cite{MR395578} \cite{MR0456220} \cite{MR692312} for $\beta <4\pi$ using non-probabilistic methods.
The extension to $\beta <8\pi$ was achived subsequently in \cite{MR4528973}.
The probabilistic point of view leads to the theory of Gaussian multiplicative chaos,
which was initiated in \cite{MR829798}.
With our scaling of the Gaussian free field in \eqref{eq:scaling-gff},
the particular value $\beta = 4\pi$ marks the end of the $L^2$ phase of the Gaussian multiplicative chaos,
where this object is typically easier to analyse.

A probabilistic analysis of the Liouville model was given in \cite{MR4054101} for $\beta<4\pi$ and subsequently in \cite{MR4238209},
where the results are extended to $\beta \in (0,8\pi)$.
Both works construct the measure $\nu^\Lv$ as an invariant measure of a certain singular stochastic partial differential equation,
also known as the stochastic quantization equation.
We remark that the stoachstic differential equations originating from the Polchinski renormalisation group approach can also be viewed as regularisation of a stochastic partial differential equation,
which is however different to the stochastic quantisation equation.
The analysis in \cite{MR4054101} and \cite{MR4238209} requires estimates of the multiplicative chaos when seen as a distribution rather than a Borel measure,
which also enter in the proof of Theorem \ref{thm:coupling-pphi-to-gff-eps}.
The reason for the restriction to $\beta <4\pi$ in the present work is ultimately the same as in \cite{MR4054101}.
The key observation in \cite{MR4238209},
which allows to extend the results to the $L^1$ phase,
is to consider the regularity of the Wick exponential under the distribution $\nu^\Lv$ rather than $\nu^\GFF$.
Generalising several technical estimates to our setting, we believe that this idea also improves our method, thereby extending Theorem \ref{thm:coupling-pphi-to-gff-eps} to $\beta<4\pi$ for both $\cE = \{ \Lv, \ShG \}$. 

Additional aspects of the Liouville model, again from a stochastic quantisation perspective,
were studied in \cite{MR4415393} \cite{ALBEVERIO2012602} \cite{MR4054101} \cite{devecchi2022singularintegrationpartsformula} and more recently in \cite{HofstetterZeitouni2025Liouville}.
All these results rely on the fact that the exponential function,
which appears as the derivative of the potential, is positive and increasing.
For the sinh-Gordon model, the analogue appearing is a hyperbolic sine,
which does not have a definite sign.
This adds an additional difficulty to the model, as most of the techniques crafted for the Liouville model cannot be used for the analysis of the sinh-Gordon model.
From a probabilistic point of view the sinh-Gordon measure is more complicated than the Liouville measure,
as it involves two non-independent multiplicative chaoses constructed from the same Gaussian field.
In the proof of Theorem \ref{thm:coupling-pphi-to-gff-eps} for $\cE = \ShG$,
we eliminate both difficulties simultaneously using in a crucial way that the hyperbolic sine is monotonically increasing. 

Recent works on the sinh-Gordon model include \cite{BarashkovDeVecchi2021Elliptic},
where the sinh-Gordon measure for $\beta < 4\pi$ using again stochastic quantisation techniques.
An important and difficult open problem for the sinh-Gordon model, not addressed by the present work, is the construction of a massless limit,
which is formally given by the limit of \eqref{eq:shg-density} with $m=0$ as $\epsilon\to 0$.
The main difficulties come from the Gaussian reference measure being ill-defined when $m=0$.
Partial progress on this problem has been achieved in \cite{Guillarmou2024sinhGordon} and \cite{Barashkov2024SmallDeviations}.


\subsection{Notation}

We write $\Lv$ and $\ShG$ to refer to the Liouville model and the sinh-Gordon model.
To avoid repeating arguments that are similar for $\Lv$ and $\ShG$,
we use $\cE$ in statement that apply for both $\Lv$ and $\ShG$ likewise.

We further use $\EE_c$ to denote expectations with respect to centred Gaussian random variables with covariance matrix $c$.

To discuss asymptotic behaviour, we use the standard big-$O$ notation and further write $\lesssim$, when an estimate holds up to a deterministic constant.
To emphasise that the dependence of a parameter $a$ say, we also write $\lesssim_a$.

We mostly study random fields with values in $X_\epsilon$.
Note that any $f\in X_\epsilon$ admits a Fourier series
\begin{equation}
f(x) = \sum_{k\in \Omega_\epsilon^*} \hat f(k) e^{ik\cdot x}, \qquad x\in \Omega_\epsilon,
\end{equation}
where $\Omega_\epsilon^* = \{ k \in 2\pi \Z^2 \colon -\pi/\epsilon \leq k_i <\pi/\epsilon  \}$ is the Fourier dual of $\Omega_\epsilon$ and $\hat f(k)\in \C$ denotes the $k$-th Fourier coefficients defined by
\begin{equation}
\hat f(k) = \epsilon^2 \sum_{x\in \Omega_\epsilon} f(x) e^{-ik\cdot x}.
\end{equation}
For $\alpha \in \R$, we define the discrete Sobolev norm through
\begin{equation}
\| f \|_{H^\alpha(\Omega_\epsilon)}^2 = \sum_{k\in \Omega_\epsilon} (1+ |k|^2| )^\alpha |\hat f(k)|^2 .
\end{equation}
When it is clear from the context, we omit the dependence on $\epsilon>0$ and write $H^\alpha (\Omega_\epsilon) \equiv H^\alpha$.
We further note that, for $f,g\in \Omega_\epsilon$,
\begin{equation}
\Big| \int_{\Omega_\epsilon} f(x) g(x) dx \Big| \leq \|f \|_{H^{\alpha}(\Omega_\epsilon)} \| g\|_{H^{-\alpha}(\Omega_\epsilon)},
\end{equation}
which follows from the orthogonality of the discrete exponentials.

\section{Stochastic control formulation}

The main tool for the proof of Theorem \ref{thm:coupling-pphi-to-gff-eps} is a variational formulation for the partition function of the measures \eqref{eq:lv-density} and \eqref{eq:shg-density} and its closely related representation as a process indexed by scale,
which comes from the Polchinski renormalisation group approach.
In this section, we briefly introduce the two concepts and state the key estimates that are used to deduce the Sobolev norm estimates in Theorem \ref{thm:coupling-pphi-to-gff-eps}.

\subsection{Polchinski renormalisation group approach and variational formulation}

In what follows, we denote by $v_0^{\cE_\epsilon}$ the potential of the measure \eqref{eq:lv-density} and \eqref{eq:shg-density},
i.e., for concreteness, we have
\begin{equation}
v_0^{\cE_\epsilon}(\phi) = \int_{\Omega_\epsilon} \wick{V(\phi_x)}_\epsilon dx ,
\end{equation}
where $V(\phi_x)= \exp(\sqrt{\beta}\phi_x)$ for $\cE = \Lv$ and $V(\phi_x)= \cosh(\sqrt{\beta}\phi_x)$ for $\cE = \ShG$.
With this notation the partition function of the measures $\nu^{\cE_\epsilon}$ is given by
\begin{equation}
Z^{\cE_\epsilon} = - \log \EE_{c_\infty^\epsilon}[e^{-v_0^{\cE_\epsilon}(\zeta)}].
\end{equation}
For $t\geq 0$, let $v_t^{\cE_\epsilon}$ be the renormalised potential defined through
\begin{equation}
\label{eq:lvshg-renormalised-potential}
e^{-v_t^{\cE_\epsilon} (\phi)} = \EE_{c_t^\epsilon} [e^{-v_0^{\cE_\epsilon}(\phi + \zeta)}], \qquad c_t^\epsilon = \int_0^t \dot c_s^\epsilon ds
\end{equation}
and note that $v_t^{\cE_\epsilon} \colon X_\epsilon \to \R$.
We remark that $t\geq 0$ is to be viewed as a scale parameter,
and that $c_t^\epsilon$ is the covariance of the small scale field of the Gaussian free field.
Thus, the expectation value in \eqref{eq:lvshg-renormalised-potential} can be seen as an averaging of the partition function of the original measure only over the short scale part of the reference measure.
This gives $v_t^{\cE_\epsilon}$ the interpretation of a potential of a renormalised version of $\nu^{\cE_\epsilon}$.
For more details we refer to \cite[Section 2.1]{MR4303014} and \cite[Section 2.1]{HofstetterZeitouni2025Liouville}.

An important quantity throughout this article is the gradient of $v_t^{\cE_\epsilon}$ with respect to the field $\phi\in X_\epsilon$.
Note that by differentiating \eqref{eq:lvshg-renormalised-potential} we have
\begin{equation}
\label{eq:renormalised-gradient-ratio-of-expectations}
\nabla v_t^{\cE_\epsilon} (\phi) 
=
\frac{\EE_{c_t^\epsilon}[\nabla v_0^{\cE_\epsilon}(\phi + \zeta)e^{-v_0^{\cE_\epsilon}(\phi + \zeta)}]}{\EE_{c_t^\epsilon}[e^{-v_0^{\cE_\epsilon}(\phi+\zeta)}]}.
\end{equation}
Here, the gradient with respect to the field is understood in the Fr\'echet sense with respect to the normalised inner product on $X_\epsilon$.
In particular, we have, for $x\in \Omega_\epsilon$
\begin{equation}
\partial_{\phi_x} v_0^{\cE_\epsilon} (\phi) = \sqrt{\beta} \wick{\exp(\sqrt{\beta} \phi_x)}_\epsilon .
\end{equation}

For $\cE= \{ \Lv, \ShG \}$, we have that the renormalised potential satisfies the a high-dimensional partial differential equation, to which we refer as the Polchinski-PDE, i.e., for $\epsilon>0$ and $v_t^{\cE_\epsilon}$ as in \eqref{eq:lvshg-renormalised-potential}, we have
\begin{align}
\label{eq:polchinski-pde}
\partial_t v_t^{\cE_\epsilon} &= \frac{1}{2} \Delta_{\dot c_t^\epsilon} v_t^{\cE_\epsilon} - \frac{1}{2} (\nabla v_t^{\cE_\epsilon})_{\dot c_t^\epsilon}^2 ,
\end{align}
where for a function $f\colon X_\epsilon \to \R$ the scale dependent differential operators in \eqref{eq:polchinski-pde} are defined by
\begin{align}
\Delta_{\dot c_t^\epsilon} f
&= \epsilon^4 \sum_{x,y} \dot c_t^\epsilon (x,y) \partial_{\phi_x} \partial_{\phi_y} f
\nnb
(\nabla f)_{\dot c_t^\epsilon}^2
&=
\epsilon^4 \sum_{x,y} \dot c_t^\epsilon (x,y) (\partial_{\phi_x}f ) ( \partial_{\phi_y} f ) .
\end{align}

The Polchinski-PDE corresponds to the backward stochastic differential equation, henceforth referred to as Polchinski-SDE,
\begin{equation}
\label{eq:polchinski-backward-sde-infty}
d\Phi_t^\epsilon = - \dot c_t^\epsilon \nabla v_t^{\cE_\epsilon}(\Phi_t^\epsilon) dt
+ (\dot c_t^\epsilon)^{1/2} dW_t^\epsilon,
\qquad \Phi_\infty = 0,
\end{equation}
which can be formally solved to obtain processes
$(\Phi_t^{\GFF_\epsilon})_{t\geq 0}$, $(\Phi_t^{\cE_\epsilon})_{t\geq 0}$ and $(\Phi_t^{\Delta_\epsilon})_{t\geq 0}$ such that for every $t\geq 0$
\begin{equation}
\Phi_t^{\cE_\epsilon} = \Phi_t^{\Delta_\epsilon} + \Phi_t^{\GFF_\epsilon},
\end{equation}
where $\Phi^{\GFF_\epsilon}$ is as in \eqref{eq:gff-process-eps},
the process $\Phi^{\cE_\epsilon} $ is non-Gaussian with $\Phi_0^{\cE_\epsilon}$ being distributed as the $\epsilon$-regularised Liouville field and sinh-Gordon field respectively,
and $\Phi_t^{\Delta_\epsilon} $ is a difference field.
To discuss the convergence of the processes as $\epsilon \to 0$, 
we further assume that all Brownian motions $(W^\epsilon)_{\epsilon>0}$ are obtained from the same cylindrical Brownian motion $W$.
For more details on this construction we refer to \cite[Section 3]{MR4399156}.
The field $\Phi_t^{\Delta_\epsilon}$ is different for $\cE= \Lv$ and $\cE = \ShG$,
but in order to simplify the notation, we omit this dependence from the notation.

The next results make this construction explicit for the measures \eqref{eq:lv-density} and \eqref{eq:shg-density}.
We give its proof in Appendix \ref{app:proof-sde-existence}.
For $\cE= \Lv$, this was achieved in \cite[Theorem 3.1]{HofstetterZeitouni2025Liouville} using Picard iteration even for $\epsilon=0$.
This method does not apply when $\cE= \ShG$, for which we use a different argument in that case, which only gives existence for $\epsilon>0$.

\begin{proposition}
\label{prop:lvshg-coupling-eps}
Let $\beta \in (0,8\pi)$. For $\epsilon>0$ and $\cE \in \{\Lv, \ShG\}$, there are unique $\cF^t$-adapted processes $\Phi^{\cE_\epsilon} \in C_0([0,\infty), X_\epsilon)$ such that for all $t\geq 0$
\begin{equation}
\label{eq:lvshg-coupling-eps}
\Phi_t^{\cE_\epsilon}
= - \int_t^\infty \dot c_s^\epsilon \nabla v_s^{\cE_\epsilon}(\Phi_s^{\cE_\epsilon}) \, ds
+ \Phi_t^{\GFF_\epsilon}
\end{equation}
and $\Phi_t^{\cE_\epsilon}$ is independent of $\Phi_0^{\GFF_\epsilon} - \Phi_t^{\GFF_\epsilon}$.
In addition, we have,
for $\cE = \Lv$,
that $\Phi_t^{\Lv_\epsilon} - \Phi_t^{\GFF_\epsilon} \leq 0$,
while, for $\cE = \ShG$,
we have $\E[\Phi_t^{\ShG_\epsilon}]=0$ for all $t\geq 0$.
\end{proposition}

For the following result, we recall the notion of the renormalised measure $\nu_t^{\cE_\epsilon}$ on $X_\epsilon$,
which is defined by
\begin{equation}
\label{eq:lvshg-renormalised-measure}
\E_{\nu_t^{\cE_\epsilon}} \big[  F \big] = e^{v_0^{\cE_\epsilon}(0)} \EE_{c_\infty^\epsilon - c_t^\epsilon} \big[ F(\zeta) e^{-v_t^{\cE_\epsilon}(\zeta)} \big].
\end{equation}
As shown in \cite[Proposition 2.1]{MR4303014} the density is given by
\begin{equation}
\label{eq:lvshg-density-renormalised-measure}
d\nu_t^{\cE_\epsilon} (\phi) \propto e^{-v_t^{\cE_\epsilon}(\phi)} d\nu_t^{\GFF_\epsilon}(\phi),
\end{equation}
where $\nu_t^{\GFF_\epsilon}$ is the law of the centred Gaussian measure on $X_\epsilon$ with covariance $c_\infty^\epsilon - c_t^\epsilon$.

\begin{proposition}
\label{prop:law-polchinski-sde}
Let $\epsilon>0$ and $\cE \in \{ \Lv, \ShG \}$ and let $\Phi^{\cE_\epsilon}$ be the solutions of \eqref{eq:lvshg-coupling-eps}.
Then $\Phi_t^{\cE_\epsilon} \sim \nu_t^{\cE_\epsilon}$ for all $t\geq 0$.
\end{proposition}

The proof uses standard arguments from the theory of stochastic differential equations as well as additional ideas to overcome the fact that the coefficients of \eqref{eq:polchinski-backward-sde-infty} are not globally Lipshitz continuous.

We further record the following result on the density $\nu_t^{\cE_\epsilon}$ defined in \eqref{eq:lvshg-renormalised-measure},
which shows that the convexity of the potential $v_0^{\cE_\epsilon}$
implies convexity of the renormalised potential $v_t^{\cE_\epsilon}$ along all scales $t\geq 0$.

\begin{lemma}
\label{lem:lvshg-he-v-pos-definite}
For all $t\geq 0$ we have
\begin{equation}
\He v_t^{\cE_\epsilon} \geq 0.
\end{equation}
In particular $v_t^{\cE_\epsilon}$ is convex.
\end{lemma}

\begin{proof}
For $\cE= \Lv$ the result is proved in \cite[Lemma 2.11]{HofstetterZeitouni2025Liouville}.
Since the argument only uses the convexity of $v_0^{\cE_\epsilon}$, the same conclusion holds for $\cE= \ShG$. 
\end{proof}

Next, we discuss the variational formulation of the partition function of $\nu^{\cE_\epsilon}$,
which uses the ideas in \cite{MR1745333} and \cite{MR1675051}, 
and its relation to Proposition \ref{prop:lvshg-coupling-eps} and Proposition \ref{prop:law-polchinski-sde}.
To this end, we continue using the notions of \cite[Section 3]{MR4665719}:
\begin{itemize}
\item For $t>0$,
\begin{equation}
\label{eq:Y-definition}
Y_t^\epsilon = \Phi_0^{\GFF_\epsilon} - \Phi_t^{\GFF_\epsilon} 
\end{equation}
denotes the small scales of the Gaussian process. In particular, $\cov(Y_t^\epsilon) = c_t^\epsilon$.
\item $\cF^t$ is the backward filtration generated by the cylindrical Brownian $W$ motion that drives the Polchinski SDE, i.e.\ $\cF^t = \sigma\big ( \{  W_s \colon s\geq t \}  \big)$.
\item $\HH_a$ denotes the space of progressively measurable (with respect to the backward filtration $(\cF^t)_{t\geq 0}$ processes which are a.s.\ in $L^2(\R^+ \times \Omega_\epsilon)$, i.e.,
$u\in \HH_a$ if and only if $u\mid_{[t,\infty)}$ is $\cB([t,\infty))\otimes \cF^t$ measurable for every $t\geq 0$ and
\begin{equation}
\int_0^\infty \| u_s\|_{L^2(\Omega_\epsilon)}^2 ds < \infty \qquad \text{a.s.,}
\end{equation}
where $\cB([t,\infty))$ denotes the Borel $\sigma$-algrebra on $[t,\infty)$.
For $t\geq 0$ we write $\HH_a[0,t]$ for the restriction of $\HH_a$ to processes on $[0,t]$ and use the convention $\HH_a[0,\infty] = \HH_a$.
We refer to elements in $\HH_a$ and $\HH[0,t]$ as drifts.
\item For $u\in \HH_a$ and $0\leq s \leq t \leq \infty$, $I_{s,t}(u)$ is the integrated drift given by
\begin{align}
I_{s,t}^\epsilon(u) &= \int_s^t q_\tau^\epsilon u_\tau d\tau, \qquad q_\tau^\epsilon = (\dot c_\tau^\epsilon \big)^{1/2},
\end{align}
with the convention $I_{0,t}^\epsilon(u) = I_t^\epsilon(u)$.
\end{itemize}

\begin{proposition}
\label{prop:conditional-boue-dupuis}
Let $\epsilon>0$ and let $\beta\in (0,8\pi)$ and let $\cE\in \{\Lv, \ShG\}$.
Then, for $t\in [0, \infty]$,
\begin{equation}
\label{eq:conditional-boue-dupuis}
 -\log \E \big [e^{-v_0^{\cE_\epsilon}(Y_t^\epsilon +  \Phi_t^{\cE_\epsilon})} \bigm| \cF^t\big]	 
=
 \inf_{u \in \HH_a[0,t]} \E \Big[v_0^{\cE_\epsilon} \big (Y_t^\epsilon + \Phi_t^{\cE_\epsilon} +  I_t^\epsilon(u)\big) + \frac{1}{2} \int_0^t \| u_s\|_{L^2}^2 ds \bigm| \cF^t\Big] \qquad \text{a.s.}
\end{equation}
\end{proposition}

\begin{proof}
Since $Y_t^\epsilon$ is independent of $\Phi_t^{\cE_\epsilon}$ we have by standard properties of conditional expectation 
\begin{align}
\label{eq:exponential-moment-conditioned-on-large-scales}
-\log \E\big[ e^{-v_0^{\cE_\epsilon}(Y_t^{\epsilon} +\Phi_t^{\cE_\epsilon})} \bigm | \cF^t \big]
= -\log \E\big[e^{-v_0^{\cE_\epsilon}(Y_t^{\epsilon}+\phi)}\big]_{\phi=\Phi_t^{\cE_\epsilon}},
\end{align}
and thus, it suffices to show \eqref{eq:conditional-boue-dupuis} for a deterministic $\phi\in X_\epsilon$.
The function $v_0^{\cE_\epsilon} \colon X_\epsilon \to \R$ is unbounded, and thus, the classical results \cite[(1.1)]{MR1745333} and \cite[Theorem 3.1]{MR1675051} do not apply.
Instead, we use the extension proved in \cite[Theorem 1.1]{MR4389160}.
The same conclusion can also be extracted from \cite[Theorem 7]{MR3255483}.
To apply \cite[Theorem 1.1]{MR4389160}, we need to verify the assumptions \cite[(A1) and (A2)]{MR4389160}.
The first assumption (A1) in our case is
\begin{equation}
\E[e^{-v_0^{\cE_\epsilon} (Y_t^\epsilon + \phi)}] < \infty ,
\end{equation}
which holds since $v_0^{\cE_\epsilon}$ is bounded below.
The second assumption reads 
\begin{equation}
\E [v_0^{\cE_\epsilon} (Y_t^\epsilon + \phi)] < \infty,
\end{equation}
which follows from the existence of exponential moments for Gaussian random variables.

Thus, we have that, for a deterministic $\phi\in X_\epsilon$,
\begin{align} 
\label{bd1}
-\log \E\big[ e^{-v_0^{\cE_\epsilon}(Y_t^\epsilon+\phi)} \big] 
&=
\inf_{u \in \mathbb{H}_{a}[0,t]} \E\Big[ v_0^{\cE_\epsilon}\big(Y_t^\epsilon+\phi+I_t^\epsilon(u)\big)+\frac{1}{2}\int_{0}^{t}\|u_s\|^{2}_{L^2(\Omega_{\epsilon})} ds\Big],
\end{align}
which, together with \eqref{eq:exponential-moment-conditioned-on-large-scales}, implies \eqref{eq:conditional-boue-dupuis}.
\end{proof}

\subsection{Existence and estimates for minimisers}

The following results give the existence of a minimiser for $\cE \in \{ \Lv, \ShG \}$ by construcing an explicit drift $\lvshgmin^\epsilon$,
which satisfies the variational formula \eqref{eq:conditional-boue-dupuis},
as well as well as estimates on this particular minimiser.
The proofs of these statements are different for $\cE = \Lv$ and $\cE = \ShG$,
and are therefore presented in Section \ref{ssec:lv-small-scales-drift} and Section \ref{ssec:shg-small-scales-drift}.

\begin{proposition}
\label{prop:bd-polchinski-correspondence}
Let $\beta \in (0,8\pi)$ and $\epsilon > 0$.
Let $\Phi^{\cE_\epsilon} \in C_0([0,\infty), X_\epsilon)$ be the unique strong solution to \eqref{eq:lvshg-coupling-eps} and let $\lvshgmin^\epsilon\colon [0,\infty\times\Omega_\epsilon \to \R$ denote the process defined by
\begin{equation}
\label{eq:lvshg-minimiser}
\lvshgmin_s^\epsilon
=
-q_s^\epsilon \nabla v_s^{\cE_\epsilon} (\Phi_s^{\cE_\epsilon}),
\qquad s\in [0,\infty).	
\end{equation}
Then $\lvshgmin^\epsilon|_{[0,t]}$ is a minimiser of the 
variational formula \eqref{eq:conditional-boue-dupuis}.
In particular, the relation between $\lvshgmin^\epsilon$ and the difference field $\Phi^{\Delta_\epsilon}$ is given by
\begin{equation}
\Phi_t^{\Delta_\epsilon} = \int_t^\infty q_s^\epsilon \lvshgmin_s^\epsilon ds.
\end{equation}
\end{proposition}


\begin{proposition}
\label{prop:lvshg-all-scales-drift}
Let $\beta \in (0,8\pi)$ and $\genmin^\epsilon \in \HH_a[0,t]$ be a minimiser of \eqref{eq:conditional-boue-dupuis}.
Then there is a constant $C>0$, which is independent of $t\geq 0$,
such that
\begin{equation}
\sup_{\epsilon>0} \E\Big[ \int_0^t \| \genmin_s^\epsilon \|_{L^2(\Omega_\epsilon)}^2 ds \Big] \leq C.
\end{equation}
\end{proposition}

In the heart of the proof of Theorem \ref{thm:coupling-pphi-to-gff-eps} is the following estimate on the small scales of minimising drifts $\genmin^\epsilon$,
which is then leveraged to estimates on the Sobolev norm of the difference field $\Phi_t^{\Delta_\epsilon}$.
Again, the proof of this result is different for $\cE= \Lv$ and $\cE= \ShG$,
and is therefore presented in Section \ref{ssec:lv-small-scales-drift} and Section \ref{ssec:shg-small-scales-drift} respectively.

\begin{proposition}
\label{prop:lvshg-small-scales-drift}
Let $\beta \in (0,4\pi)$ and let $\cE \in \{ \Lv, \ShG\}$.
Let $t>0$ and 
let $\genmin^\epsilon$ be a minimiser of \eqref{eq:conditional-boue-dupuis}.
For any $\delta \in (0,1-\beta/4\pi)$ there are positive random variable $\cW_{\delta,t}^{\cE_\epsilon}$ satisfying $\sup_{\epsilon>0} \sup_{t\geq 0}  \E[ \cW_{\delta,t}^{\cE_\epsilon}] < \infty$,
such that
\begin{equation}
\label{eq:lvshg-small-scales-drift}
\E\Big[\int_0^t \|\genmin_s^\epsilon \|_{L^2}^2 d s \bigm| \cF^t \Big]^{1/2}
\leq t^{\delta/2} \cW_{\delta, t}^{\cE_\epsilon}.
\end{equation}
\end{proposition}

\section{Existence of minimisers: proof of Proposition \ref{prop:bd-polchinski-correspondence}}

The argument is based on an application of Ito's formula to a function of the solution to \eqref{eq:polchinski-backward-sde-infty}.
In the course of the proof, we use the following result,
which allows to bound the $L^1$ norm of $\nabla v_t^{\cE_\epsilon}(\phi)$, defined by
\begin{equation}
\| \nabla v_t^{\cE_\epsilon}(\phi) \|_{L^1(\Omega_\epsilon)}
=
\epsilon^2 \sum_{x\in \Omega_\epsilon} |\partial_{\phi_x} v_t^{\cE_\epsilon}(\phi) |.
\end{equation}
We emphasise that this estimate is uniform in $\epsilon>0$.

\begin{proposition}
\label{prop:gradient-l1-upper-bound}
Let $\beta \in (0,8\pi)$ and let $\epsilon >0$. Then, for $\phi \in X_\epsilon$, 
\begin{equation}
\| \nabla v_t^{\Lv_\epsilon}(\phi) \|_{L^1(\Omega_\epsilon)}
\lesssim_m
\sqrt{\beta} \int_{\Omega_\epsilon} L_t^{\beta/4\pi}  e^{\sqrt{\beta}\phi_x} dx.
\end{equation}
Similarly, we have
\begin{equation}
\| \nabla v_t^{\ShG_\epsilon} (\phi) \|_{L^1(\Omega_\epsilon)}
\lesssim_m 
\sqrt{\beta} \int_{\Omega_\epsilon} L_t^{\beta/4\pi}  \cosh\big(\sqrt{\beta} \phi_x \big) dx.
\end{equation}
\end{proposition}

\begin{proof}
Apart from one estimate the argument is identical $\cE = \Lv$ and $\cE= \ShG$,
for which we treat both cases simultaneously.
In line with the general convention, we write $v_t^{\cE_\epsilon}$ and $\nabla v_t^{\cE_\epsilon}$ in statements that hold for $\cE \in \{\Lv, \ShG \}$
and explicitely write $\Lv$ and $\ShG$ otherwise.
From \eqref{eq:renormalised-gradient-ratio-of-expectations} we obtain, for $\phi \in X_\epsilon$,
\begin{equation}
\label{eq:gradient-vt}
\| \nabla v_t^{\cE_\epsilon} (\phi)\|_{L^1(\Omega_\epsilon)} 
\leq
\frac{\EE_{c_t^\epsilon}[\| \nabla v_0^{\cE_\epsilon}(\phi + \zeta) \|_{L^1(\Omega_\epsilon)} e^{-v_0^{\cE_\epsilon}(\phi + \zeta)}]}{\EE_{c_t^\epsilon}[e^{-v_0^\cE(\phi+\zeta)}]}.
\end{equation}
Now, for $\cE = \Lv$, we have
\begin{equation}
\label{eq:lv-nablavx}
|\partial_{\phi_x} v_0^{\Lv_\epsilon} (\phi)|
=
\epsilon^{\beta/4\pi} \sqrt{\beta} \exp(\sqrt{\beta} \phi_x )
= \sqrt{\beta} \wick{\exp(\sqrt{\beta}\phi_x ) }_\epsilon,
\end{equation}
while for $\cE = \ShG$,
we use the estimate $|\sinh(\sqrt{\beta} \phi_x)| \leq \cosh(\sqrt{\beta} \phi_x)$ to obtain 
\begin{equation}
\label{eq:shg-nablavx}
|\partial_{\phi_x} v_0^{\ShG_\epsilon} (\phi) | 
=
|\epsilon^{\beta/4\pi} \sqrt{\beta} \sinh(\sqrt{\beta} \phi_x )| \leq \epsilon^{\beta/4\pi} \sqrt{\beta} \cosh(\sqrt{\beta} \phi_x )
= \sqrt{\beta} \wick{\cosh(\sqrt{\beta}\phi_x ) }_\epsilon.
\end{equation}
Summing \eqref{eq:lv-nablavx} and \eqref{eq:shg-nablavx} over $x\in \Omega_\epsilon$, we see that
\begin{equation}
\|\nabla v_0^{\cE_\epsilon} (\phi) \|_{L^1} \leq \sqrt{\beta} v_0^{\cE_\epsilon}(\phi),
\end{equation}
and thus, we obtain from \eqref{eq:gradient-vt}
\begin{equation}
\| \nabla v_t^{\cE_\epsilon}(\phi) \|_{L^1} 
\leq
\sqrt{\beta} \frac{\EE_{c_t^\epsilon}[ v_0^{\cE_\epsilon}(\phi+\zeta)  e^{-v_0^{\cE_\epsilon}(\phi + \zeta)}]}{\EE_{c_t^\epsilon}[e^{-v_0^{\cE_\epsilon}(\phi+\zeta)}]}.
\end{equation}
For $C>0$, we split the expectation on the right hand side of the previous display into $\{v_0^{\cE_\epsilon}(\phi + \zeta) \leq C\}$ and $\{v_0^{\cE_\epsilon}(\phi+\zeta)>C\}$ and obtain 
\begin{equation}
\| \nabla v_t^{\cE_\epsilon} (\phi) \|_{L^1} 
\leq
\sqrt{\beta} \Big[ C +  \frac{\EE_{c_t^\epsilon}[ v_0^{\cE_\epsilon}(\phi+\zeta) ] e^{-C}} {\EE_{c_t^\epsilon}[e^{-v_0^{\cE_\epsilon}(\phi+\zeta)}]} \Big].
\end{equation}
To further estimate the last display we choose $C= C(t,\phi)$ according to
\begin{equation}
e^{-C} = \EE_{c_t^\epsilon}[e^{-v_0^{\cE_\epsilon}(\phi+\zeta)}] 
\quad \iff \quad
C = -\log \EE_{c_t^\epsilon}[e^{-v_0^{\cE_\epsilon}(\zeta + \phi)}] = v_t^{\cE_\epsilon}(\phi).
\end{equation}
Then, using \eqref{eq:ct-asymptotics}
we have by Jensen's inequality
\begin{equation}
\| \nabla v_t^{\Lv_\epsilon}(\phi) \|_{L^1} 
\leq 
\sqrt{\beta} \Big[  v_t^{\Lv_\epsilon}(\phi) + \EE_{c_t^\epsilon}[v_0^{\cE_\epsilon}(\phi + \zeta)]  \Big]
\lesssim_{m}
2 \sqrt{\beta} \int_{\Omega_\epsilon} L_t^{\beta/4\pi}\exp(\sqrt{\beta}\phi_x) dx,
\end{equation}
and similarly
\begin{equation}
\label{eq:l1-vt}
\| \nabla v_t^{\ShG_\epsilon}(\phi) \|_{L^1} 
\lesssim_m 
2 \sqrt{\beta} \int_{\Omega_\epsilon} L_t^{\beta/4\pi} \cosh(\sqrt{\beta}\phi_x)  dx.
\end{equation}
\end{proof}

Next, we give the proof of Proposition \ref{prop:bd-polchinski-correspondence} treating again $\cE \in  \{\Lv, \ShG\}$ simultaneously.

\begin{proof}[Proof of Proposition \ref{prop:bd-polchinski-correspondence}]
To simplify the notation, we drop $\epsilon>0$ throughout this proof.
Since $(\Phi_t^{\cE})_{t\in [0,\infty]}$ is $\cF^t$-measurable and continuous,
it follows that $\lvshgmin^\epsilon |_{[0,t]} \in \HH_a^\epsilon[0,t]$ for any $t\in [0,\infty]$.

Applying It\^o's formula to $v_t^\cE(\Phi_t^\cE)$ we obtain
\begin{align}
d v_t^{\cE}(\Phi_t^{\cE}) &= - \nabla v_t^{\cE}(\Phi_t^{\cE}) \dot c_t \nabla v_t^{\cE}(\Phi_t^{\cE}) dt - \partial_t v_t^{\cE}(\Phi_t^{\cE}) dt
\nnb 
& +\frac{1}{2} \tr\big(\He v_t^{\cE} \dot c_t \big) dt
+\nabla v_t^{\cE}(\Phi_t^{\cE}) \dot c_t^{1/2} dW_t \nnb
&= -\big(\nabla v_t^{\cE} (\Phi_t^{\cE})\big)_{\dot c_t}^2 - \partial_t v_t^{\cE}(\Phi_t^{\cE}) dt
+
\frac{1}{2} \Delta_{\dot c_t} v_t^{\cE}(\Phi_t^{\cE}) dt
+ \nabla v_t^{\cE}(\Phi_t^{\cE}) \dot c_t^{1/2} dW_t
\nnb
&= -\frac{1}{2} \big(\nabla v_t^{\cE} (\Phi_t^{\cE})\big)_{\dot c_t}^2 dt + \nabla v_t^{\cE}(\Phi_t^{\cE}) \dot c_t^{1/2} dW_t,
\end{align}
where we used the Polchinski equation \eqref{eq:polchinski-pde} in the last step.
Note that $\Phi^{\cE}$ is the solution to a backward SDE starting at $t=\infty$.
To justify the application of It\^o's formula in this case,
we fix $T>0$ and consider,
for $\tau \leq T$, the process $\tilde \Phi^\cE$ solving the forward SDE
\begin{equation}
\label{eq:lvshg-forward-SDE-to-T-minimiser-proof}
d \tilde \Phi_\tau = - \dot c_{T-\tau} \nabla v_{T-\tau}^{\cE} (\tilde \Phi_\tau) d\tau + \dot c_{T-\tau}^{1/2} d W_t^\epsilon \qquad 0\leq \tau\leq T,
\qquad  \tilde \Phi_0 \sim \Phi_T^{\cE}.
\end{equation}
Note that, since $\Phi_T^{\cE}$ is independent of $\sigma(W_\tau \colon \tau \leq T)$,
the SDE \eqref{eq:lvshg-forward-SDE-to-T-minimiser-proof} is well-defined.
We then apply the standard It\^o formula to $v_{T-\tau}^{\cE} (\tilde \Phi_\tau^{\cE})$ for $\tau \in [0,T-t]$ and obtain
\begin{equation}
\label{eq:v-phi-after-ito-and-cond-exp}
v_0^{\cE} (\Phi_0^{\cE} ) 
= 
v_t^{\cE} (\Phi_t^{\cE} )
- \frac{1}{2} \int_0^t \big( \nabla v_s^{\cE} (\Phi_s^{\cE}) \big)_{\dot c_s}^2 ds
+ 
\int_0^t \nabla v_s^{\cE} (\Phi_s^{\cE}) \cdot \dot c_s^{1/2} dW_s.
\end{equation}
Taking the conditional expectation, we have
\begin{equation}
\label{eq:lvshg-ito-conditional-expectation-phi-T}
\E \Big[ v_0^{\cE}(\Phi_0^{\cE}) - v_t^{\cE}(\Phi_t^{\cE}) \bigm | \cF^t \Big]
= - \frac{1}{2} \E \Big[ \int_0^t \big(\nabla v_s^{\cE} (\Phi_s^{\cE})\big)_{\dot c_s}^2  ds \bigm | \cF^t \Big],
\end{equation}
where we used that for $t\in [0,\infty]$
\begin{equation}
\label{eq:int-nablav-T-ct12-dw-finite-as}
\E \Big[ \int_0^t \nabla v_s^{\cE}(\Phi_s^{\cE}) \dot c_s^{1/2} d W_s \bigm | \cF^t \Big] = 0 \qquad \text{a.s.}
\end{equation}
To see that this is true, we need to show that the integrand is in $L^2(\avg{W|_{[0,t]}})$ conditional on $\cF^t$, i.e., we need that
\begin{equation}
\E \big[ \int_0^t
\| \dot c_s^{1/2} \nabla v_s^{\cE}(\Phi_s^{\cE}) \|_{L^2(\Omega_\epsilon)}^2 ds \bigm | \cF^t \big] < \infty \qquad \text{a.s.,}
\end{equation}
where we note that
\begin{equation}
\| \dot c_s^{1/2} \nabla v_s^\cE(\Phi_s^{\cE}) \|_{L^2(\Omega_\epsilon)}^2
= \nabla v_s^{\cE} (\Phi_s^{\cE}) \dot c_s^\epsilon \nabla v_s^{\cE}(\Phi_s^{\cE}) .
\end{equation}
To prove the last claim, we show the stronger estimate
\begin{equation}
\label{eq:expectation-l2-drift-term}
\E \Big[ \int_0^t
\| \dot c_s^{1/2} \nabla v_s^{\cE}(\Phi_s^{\cE}) \|_{L^2(\Omega_\epsilon)}^2 ds \Big] < \infty .
\end{equation}
To this end, we use the crude bound $\sup_{t\geq 0} \sup_{x,y\in \Omega_\epsilon} |\dot c_t^\epsilon(x,y)|\lesssim_\epsilon 1$ and estimate
\begin{align}
\| \dot c_s^{1/2} \nabla v_s^\cE(\Phi_s^{\cE}) \|_{L^2(\Omega_\epsilon)}^2
&\lesssim_\epsilon
\epsilon^4\sum_{x,y\in \Omega_\epsilon}
|\partial_{\phi_x} v_t^{\cE} (\Phi_t^{\cE})| 
|\partial_{\phi_y} v_t^{\cE} (\Phi_t^{\cE})|
\nnb
&=
\Big(\epsilon^2 \sum_{x\in \Omega_\epsilon} |\partial_{\phi_x} v_t^{\cE} (\Phi_t^{\cE})| 
\Big)^2
= \| \nabla v_t^{\cE}(\Phi_t^{\cE}) \|_{L^1}^2 .
\end{align}
Now, Proposition \ref{prop:gradient-l1-upper-bound}, gives an upper bound on the expectation value in \eqref{eq:expectation-l2-drift-term} in terms of exponential moments of $\Phi_s^{\cE}$.
For $\cE= \Lv$ we have
\begin{equation}
\label{eq:lv-upper-bound-expectation-l2-dct12nablavt}
\E \Big[\int_0^t
\| \dot c_s^{1/2} \nabla v_s^{\Lv}(\Phi_s^{\Lv}) \|_{L^2(\Omega_\epsilon)}^2 ds \Big]
\lesssim_{\epsilon, m}
\beta \int_0^t \E \Big[\Big( \epsilon^2 \sum_{x\in \Omega_\epsilon} L_s^{\beta/4\pi} \exp\big( \sqrt{\beta}\Phi_s^{\cE} (x) \big) \Big)^2\Big] ds
 .
\end{equation}
We then use $\Phi_s^{\Lv_\epsilon} \leq \Phi_s^{\GFF_\epsilon}$,
which holds by Proposition \ref{prop:lvshg-coupling-eps} for $\cE= \Lv$,
and the existence of exponential moments of Gaussian random variables to show that the last display in \eqref{eq:lv-upper-bound-expectation-l2-dct12nablavt} is finite.
In the case $\cE = \ShG$ the same conclusion holds by Proposition \ref{lem:lvshg-he-v-pos-definite} and the Brascamp-Lieb inequality for exponential moments \cite{MR0450480}.
This concludes the proof of \eqref{eq:int-nablav-T-ct12-dw-finite-as}.

Returning to the actual proof, we first note that
\begin{equation}
e^{-v_t^\cE(\Phi_t^\cE)}  = \E\Big [e^{-v_0^\cE(Y_t + \Phi_t^\cE)} \bigm | \cF^t \Big],
\end{equation}
which holds by the indpendence of $Y_t$ and $\Phi_t^\cE$ and standard properties of conditional expectation.
Therefore, we obtain from \eqref{eq:v-phi-after-ito-and-cond-exp}
\begin{align}
\E &\Big[ \frac{1}{2}\int_0^t \big(\nabla v_s^{\cE} (\Phi_s^{\cE})\big)_{\dot c_s}^2  ds \bigm | \cF^t \Big]
=  v_t^{\cE}(\Phi_t^{\cE}) - \E \Big[ v_0^{\cE}(\Phi_0^{\cE}) \bigm | \cF^t\Big]
\nnb
&= - \log \E\Big [e^{-v_0^{\cE}(Y_t + \Phi_t^{\cE})} \bigm | \cF^t \Big] - \E \Big[  v_0^{\cE} \Big(\int_0^\infty q_t dW_t - \int_0^\infty \dot c_t \nabla v_t^{\cE}(\Phi_t^{\cE})dt \Big) \bigm |\cF^t \Big].
\end{align}
Rearranging this and using \eqref{eq:Y-definition} show that
\begin{align}
-  \log \E\Big [e^{-v_0^\cE(Y_t + \Phi_t^\cE)} \bigm | \cF^t \Big] 
&= \E \Big[   v_0^\cE\Big(\int_0^\infty q_t dW_t - \int_0^\infty \dot c_t \nabla v_t^\cE(\Phi_t^\cE)dt \Big)
+
\frac{1}{2}\int_0^t \big(\nabla v_s^\cE (\Phi_s^\cE)\big)_{\dot c_s}^2  ds \bigm | \cF^t \Big]
\nnb
& = \E \Big[ v_0^\cE \Big( Y_t + \Phi_t^\cE - \int_0^t \dot c_s \nabla v_s^\cE(\Phi_s^\cE)ds \Big)
+ \frac{1}{2}\int_0^t \big(\nabla v_s^\cE (\Phi_s^\cE)\big)_{\dot c_s}^2  ds \bigm | \cF^t \Big],
\end{align}
which proves that $\lvshgmin_s = -q_s \nabla v_s^\cE(\Phi_s^\cE)$, $s\in [0,t]$ is a minimiser for \eqref{eq:conditional-boue-dupuis}.
\end{proof}

\section{Estimates for minimising drifts: Proofs of Proposition \ref{prop:lvshg-all-scales-drift} and Proposition \ref{prop:lvshg-small-scales-drift}}

The proofs of the estimates on minimisers are different for $\cE = \Lv$ and $\cE = \ShG$.
In order to make the presentation clear, we treat both cases seperately starting with the case $\cE = \Lv$.
Before, we establish a useful estimate for the minimisers of the Boue-Dupuis formula based on a comparision with suitable competitors.
The proof relies on the monotonicity of the derivatives of $v_0^{\cE_\epsilon}$ and is similar for $\cE = \Lv$ and $\cE = \ShG$, for which the following result is stated and proved for both cases.

\begin{proposition}
\label{prop:lvshgmin-small-scales-pre-estimate}
Let $\beta \in (0,8\pi)$ and let $t>0$ and $\epsilon>0$. Let $\genmin^\epsilon \in \HH_a[0,t]$ be a minimiser of \eqref{eq:conditional-boue-dupuis}.
For $\cE = \Lv$ we have
\begin{equation}
\label{eq:lvmin-small-scales-pre-estimate}
\E \Big[ \int_0^t \|u_s^\epsilon\|_{L^2}^2 ds \bigm | \cF^t \Big]
\leq \sqrt{\beta} \E\big[ - \int_{\Omega_\epsilon} I_t^\epsilon(u^\epsilon) \wick{\exp( \sqrt{\beta} Y_t^\epsilon + \Phi_t^{\Lv_\epsilon} )}_\epsilon dx \bigm | \cF^t \big]
\qquad \text{a.s.}
\end{equation}
Similarly, we have for $\cE = \ShG$
\begin{equation}
\label{eq:shgmin-small-scales-pre-estimate}
\E \big[ \int_0^t \|\genmin_s^\epsilon\|_{L^2}^2 ds \bigm | \cF^t \big]
\leq \sqrt{\beta} \E\Big[ - \int_{\Omega_\epsilon} I_t^\epsilon(\genmin^\epsilon) \wick{\sinh\big( \sqrt{\beta} (Y_t^\epsilon + \Phi_t^{\ShG_\epsilon})\big )}_\epsilon dx \bigm | \cF^t \Big] \qquad \text{a.s.}
\end{equation}
\end{proposition}

\begin{proof}
Throughout this proof, $\epsilon>0$ is dropped from the notation in most occasions. We also write $L^2$ instead of $L^2(\Omega_\epsilon)$.
Let $\genmin$ be the minimiser, which exists by the assumption.
Since $\genmin \in \mathbb{H}_a[0,t]$,
we have that a.s.\
\begin{equation}
\int_0^t \|\genmin_s\|_{L^2}^2 ds < \infty.
\end{equation}
For $h>0$ we define
\begin{equation}
\genmin^{h,-} = \genmin(1-h).
\end{equation}
Since $\genmin$ is adapted, so is $\genmin^{h,-}$.
Moreover, we have
\begin{equation}
\label{eq:perturbed-minimiser-finite-l2}
\int_0^t \|\genmin_s^{h,-}\|_{L^2}^2 ds
= (1 - h)^2 \int_0^t   \|\genmin_s\|_{L^2}^2 ds  < \infty,
\end{equation}
which shows that $\genmin^{h,-} \in \mathbb{H}_a[0,t]$.

Since $\genmin$ is a minimiser of \eqref{eq:conditional-boue-dupuis}, we have that
\begin{equation}
\E \Big[v_0^\cE \big (Y_t + \Phi_t^\cE +  I_t(\genmin)\big) + \frac{1}{2} \int_0^t \|\genmin_s\|_{L^2}^2 ds \bigm| \cF^t\Big]
\leq  \E \Big[v_0^\cE \big (Y_t + \Phi_t^\cE +  I_t(\genmin^{h,-}) \big) + \frac{1}{2} \int_0^t \|\genmin_s^{h,-}\|_{L^2}^2 ds \bigm| \cF^t\Big] ,
\end{equation}
and thus, using the equality in \eqref{eq:perturbed-minimiser-finite-l2} we have
\begin{align}
\label{eq:minimiser-competitor-estimate}
\frac{1}{2} \E \Big[ \int_0^t \|\genmin_s\|_{L^2}^2 ds \bigm| \cF^t\Big] & - 
\frac{1}{2} (1-h)^2 \E \Big[ \int_0^t \|\genmin_s\|_{L^2}^2 ds \bigm| \cF^t\Big] 
 \nnb
&\leq \E \Big[v_0^\cE \big (Y_t + \Phi_t^\cE +  I_t(\genmin^{h,-})\big) \bigm| \cF^t\Big]
- \E \Big[v_0^\cE \big (Y_t + \Phi_t^\cE +  I_t(\genmin)\big)  \bigm| \cF^t\Big].
\end{align}
Furthermore, we note that by linearity
\begin{equation}
I_t(u^{h,-}) = \int_0^t q_\tau u_\tau^{h,-}  d\tau
= (1-h) \int_0^t q_\tau^\epsilon u_\tau  d\tau ,
\end{equation}
and thus, we have
\begin{align}
\label{eq:lv-small-scales-min-and-pert-min}
(h - \frac{1}{2} h^2) &\E \Big[ \int_0^t \|\genmin_s\|_{L^2}^2 ds \bigm| \cF^t\Big] \nnb
&\leq \E \Big[v_0^\cE \big (Y_t + \Phi_t^\cE +  (1-h) I_t(\genmin)\big) 
- 
v_0^{\cE_\epsilon} \big (Y_t^\epsilon + \Phi_t^{\cE_\epsilon} +  I_t^\epsilon(\genmin) \big) \bigm| \cF^t\Big].
\end{align}

From here on, we treat the the cases $\cE = \Lv$ and $\cE= \ShG$ separately.
To estimate the last right hand side for $\cE= \Lv$, we let $\Psi,\varphi \in X_\epsilon$ and consider the function $F^\Lv \colon [0,1]\to \R$,
\begin{equation}
F^\Lv (h) = v_0^\Lv \big( \sqrt{\beta} (\Psi + (1-h)\varphi )  \big).
\end{equation}
Note that we have
\begin{equation}
\label{eq:Flv-derivative}
\frac{d}{dh} F^\Lv (h) = - \sqrt{\beta} \int_{\Omega_\epsilon} \wick{ \exp\big( \sqrt{\beta} (\Psi_x + (1-h)\varphi_x)  \big) }_\epsilon \varphi_x dx,
\end{equation}
so that we have by the mean value theorem for some $\tilde h \in (0,1)$ depending on $\Psi$ and $\varphi$,
\begin{equation}
\label{eq:Flv-mean-value-theorem}
F^\Lv(h) - F^\Lv(0) = h \frac{d}{dh} F^\Lv (\tilde h).
\end{equation}

Now, we observe that, for $a,b \in \R$,
\begin{equation}
\label{eq:exp-monotonicity}
- a e^{b + a} \leq -a e^b,
\end{equation}
and thus, we get from \eqref{eq:lv-small-scales-min-and-pert-min}, \eqref{eq:Flv-derivative} and \eqref{eq:Flv-mean-value-theorem}
\begin{align}
(1 - \frac{1}{2} h) &\E \Big[ \int_0^t \|\genmin_s\|_{L^2}^2 ds \bigm| \cF^t\Big] \nnb
&\leq \sqrt{\beta} \E \Big[ -\int_{\Omega_\epsilon} I_t(\genmin) \wick{\exp\big(\sqrt{\beta}(Y_t + \Phi_t^\cE) \big)}_\epsilon dx \bigm| \cF^t\Big].
\end{align}
The last right hand side is independent of $h$, and thus, the estimate \eqref{eq:lvmin-small-scales-pre-estimate} follows when $h\to 0$.

We now discuss the case $\cE = \ShG$.
As before we have
\begin{align}
(h - \frac{1}{2} h^2) &\E \Big[ \int_0^t \|\genmin_s\|_{L^2}^2 ds \bigm| \cF^t\Big] \nnb
&\leq \E \Big[v_0^\ShG \big (Y_t + \Phi_t^\ShG +  (1 - h ) I_t(\genmin)\big) 
- 
v_0^\ShG \big (Y_t + \Phi_t^\ShG+  I_t(\genmin)\big)  \bigm| \cF^t\Big] .
\end{align}
Now, we consider, for $\Psi, \varphi \in X_\epsilon$, the function $F^\ShG \colon [0,1] \to \R$,
\begin{equation}
F^\ShG(h) = v_0^\ShG \big(\sqrt{\beta}(\Psi + (1-h) \varphi)\big),
\end{equation}
and note that
\begin{equation}
\frac{d}{dh} F^\ShG(h) =  - \sqrt{\beta} \int_{\Omega_\epsilon} \wick{\sinh\big( \sqrt{\beta} (\Psi_x + (1-h) \varphi_x) \big) }_\epsilon \varphi_x dx .
\end{equation}
We now observe that, for all $a,b \in \R$, 
\begin{equation}
\label{eq:sinh-monotonicity}
- a \sinh(b + a) \leq -a \sinh(b).
\end{equation}
Using similar arguments as for $\cE=\Lv$, we have, for any $h \in (0,1)$,
\begin{align}
(1 - \frac{1}{2} h) &\E \Big[ \int_0^t \|\genmin_s\|_{L^2}^2 ds \bigm| \cF^t\Big] \nnb
&\leq \E \Big[
-\int_{\Omega_\epsilon} \wick{\sinh \big(\sqrt{\beta}(Y_t + \Phi_t^\ShG) \big) }_\epsilon I_t(\genmin) dx \bigm| \cF^t\Big],
\end{align}
and thus, \eqref{eq:shgmin-small-scales-pre-estimate} follows when $h\to 0$.
\end{proof}

\begin{remark}
\label{rem:identity-scmall-scale-drift}
Using standard results of measure theory, we believe it is possible to improve the above estimate and deduce the following identity for minimisers $\genmin^\epsilon$ to \eqref{eq:conditional-boue-dupuis}:
let $V(\phi)= \wick{\exp(\sqrt{\beta}\phi)}_\epsilon$ for $\cE = \Lv$ and $V(\phi) = \wick{\cosh(\sqrt{\beta}\phi)}_\epsilon$ for $\cE = \ShG$. Then
\begin{equation}
\label{eq:conjectured-identity-small-scales-drift}
\E \big[ \int_0^t \|\genmin_s^\epsilon\|_{L^2}^2 ds \bigm | \cF^t \big]
= \E\big[ - \int_{\Omega_\epsilon} I_t^\epsilon(\genmin^\epsilon) V'\big( Y_t^\epsilon + \Phi_t^{\cE_\epsilon} + I_t^\epsilon(\genmin^\epsilon) \big) dx \bigm | \cF^t \big].
\end{equation}
We note that, using \eqref{eq:exp-monotonicity} and \eqref{eq:sinh-monotonicity},
this implies the estimates in Proposition \ref{prop:lvshgmin-small-scales-pre-estimate}.
For $\beta \in (0,4\pi)$,
the estimates in Proposition \ref{prop:lvshgmin-small-scales-pre-estimate} are sufficient to establish estimates on $\Phi^\Delta$ in Theorem \ref{thm:coupling-pphi-to-gff-eps}.
The main restriction to $\beta \in (0,4\pi)$ of the method comes from the regularity of the Wick exponentials stated in Lemma \ref{lem:sobolev-norms-liouville-finite} and Lemma \ref{lem:sobolev-norms-sinh-finite}.
As shown in \cite[Corollary 2.4]{MR4238209}, it is possible to obtain regularity estimates for the Wick exponentials for $\beta \in (0,8\pi)$,
if the Wick exponential is of the field $\Phi_0^\Lv$ rather than $\Phi_0^\GFF$.
This can be achieved using the conjectured identity \eqref{eq:conjectured-identity-small-scales-drift} together with generalisation of \cite[Corollary 2.4]{MR4528973} to the present case and the observation that,
under the unconditional measure,
\begin{equation}
Y_t^\epsilon + \Phi_t^{\cE_\epsilon} + I_t^\epsilon(\lvshgmin^\epsilon) \sim \nu^{\cE_\epsilon}
\end{equation}
where $\lvshgmin^\epsilon$ is as in \eqref{eq:lvshg-minimiser}.
We therefore believe that it is possible to remove the restriction to the $L^2$ phase and prove Sobolev norm estimates as in Theorem \ref{thm:coupling-pphi-to-gff-eps} for $\beta \in (0,8\pi)$.
\end{remark}

To deduce Proposition \ref{prop:lvshg-small-scales-drift} from Proposition \ref{prop:lvshgmin-small-scales-pre-estimate} we use a duality estimate on the right hand side together with regularity estimates of the Wick exponentials.
The relation between the Sobolev norms of integrated drifts and its $L^2$ norm is given by the following non-probabilistic estimates,
which are used in the sequel.
Recall that we denoted, for $s,t\in [0,\infty]$, $\epsilon>0$ and $u \in \HH_a$,
\begin{equation}
I_{s,t}^\epsilon(u) = \int_s^t q_\tau^\epsilon u_\tau d\tau, \qquad q_\tau^\epsilon = (\dot c_\tau^\epsilon)^{1/2}
\end{equation}
with the convention $I_t^\epsilon(u) = I_{0,t}^\epsilon(u)$.
Furthermore, we note that we have, for $u\in \HH_a$,
\begin{equation}
\|I_{s,t}^\epsilon(u) \|_{H^\alpha}^2= \sum_{k\in \Omega_\epsilon^*}    (1+|k|^2)^\alpha | \widehat{ I_{s,t}^\epsilon(u) } |^2 =\sum_{k\in \Omega_\epsilon^*}    (1+|k|^2)^\alpha \Big| \int_s^t \hat q_\tau^\epsilon(k) \hat u_\tau(k) d\tau\Big|^2,
\end{equation}
where $\hat q_\tau^\epsilon(k)$, $k\in \Omega_\epsilon^*$ denote the Fourier multiplier of $q_\tau^\epsilon$. It can be shown that
\begin{equation}
\hat q_\tau^\epsilon(k) = \frac{1}{\tau \big(-\hat \Delta^\epsilon (k) + m^2\big) + 1}, \qquad k \in \Omega_\epsilon^*,
\end{equation}
where $-\hat \Delta^\epsilon (k)$ denote the Fourier multiplier of $-\Delta^\epsilon$.
We refer to \cite[Section 4.1]{MR4665719} for more details.
Using that $-\hat \Delta^\epsilon(k) \geq c |k|^2$, $k\in \Omega_\epsilon^*$, for a constant $c>0$ which is independent of $\epsilon>0$, it can further be shown that
\begin{equation}
\label{eq:I-t-infty-uniform-in-h1}
\| I_{t, \infty}^\epsilon(u) \|_{H^1(\Omega_\epsilon)}^2 \lesssim  \int_t^\infty \|u_\tau \|_{L^2(\Omega_\epsilon)}^2 d\tau .
\end{equation}
The following statements generalise this to Sobolev norms of regularity $\alpha \in [0,2]$ at the cost and gain of fractional powers of the scale parameter.
For a proof of these results, where refer to \cite[Lemma 4.3 and Lemma 4.4]{MR4665719}.

\begin{lemma}
\label{lem:sobolev-norm-integrated-drift-1}
For any $\alpha \in [0,1)$ we have
\begin{equation}
\label{eq:small-scales-integrated-drift-h-norm}
\|I_{s,t}^\epsilon(u)\|_{H^\alpha(\Omega_\epsilon)}^2
\lesssim_\alpha (t-s)^{1-\alpha} \int_0^t \|u_\tau\|_{L^2(\Omega_\epsilon)}^2 d\tau .
\end{equation}
\end{lemma}

\begin{lemma}
\label{lem:sobolev-norm-integrated-drift-2}
For any $\alpha \in (1,2]$ we have
\begin{equation}
\label{eq:h-alpha-norm-integrated-drift-small-scales-liouville}
\|I_{s,t}^\epsilon(u) \|_{H^\alpha(\Omega_\epsilon)}^2
\lesssim_\alpha \frac{t-s}{s^\alpha} \int_s^t \|u_\tau\|_{L^2(\Omega_\epsilon)}^2 d\tau .
\end{equation}
\end{lemma}

We continue with the proof the key estimates in Proposition \ref{prop:lvshg-all-scales-drift} and Proposition \ref{prop:lvshg-small-scales-drift}.
While the conclusion is the same, the argument is different for $\cE = \Lv$ and $\cE = \ShG$.
From here on, we discuss the cases $\cE= \Lv$ and $\cE= \ShG$ separately starting with $\cE= \Lv$.

\subsection{Liouville}
\label{ssec:lv-small-scales-drift}

The following observation is helpful on several occasions as it allows to reduce the discussion to a Gaussian setting.

\begin{lemma}
\label{eq:lv-int-genmin-negative}
Let $\cE= \Lv$ and $\epsilon>0$ and let $\genmin^\epsilon$ be a minimiser of \eqref{eq:conditional-boue-dupuis}.
Then, conditional on $\cF^t$, we have that $\genmin_s^\epsilon \leq 0$ for Lebesgue-a.e.\ $s\in [0,t]$.
In particular, conditional on $\cF^t$,
\begin{equation}
\label{eq:lv-int-drivt-genmin-negative}
I_t^\epsilon(\genmin^\epsilon) \leq 0.
\end{equation}
\end{lemma}

\begin{proof}
We consider the process $\mgenmin = \genmin^\epsilon \mathbf{1}_{\genmin^\epsilon\leq 0}$ and note that $\mgenmin$ is adapted and satisfies
\begin{equation}
\int_0^t \| \mgenmin_s \|_{L^2(\Omega_\epsilon)}^2 ds
\leq \int_0^t \| \genmin_s^\epsilon \|_{L^2(\Omega_\epsilon)}^2 ds
< \infty \qquad \text{a.s.}
\end{equation}
Thus, $\mgenmin \in \HH_a[0,t]$, and moreover, since $q_t^\epsilon$ has positive entries, we have
\begin{equation}
I_t^\epsilon (\mgenmin ) \leq I_t^\epsilon(\genmin^\epsilon).
\end{equation}
Since $v_0^{\Lv_\epsilon}$ is increasing, we have
\begin{align}
\E \Big[v_0^{\Lv_\epsilon} \big (Y_t^\epsilon + \Phi_t^{\Lv_\epsilon} +  I_t^\epsilon( \mgenmin )\big)
&+ \frac{1}{2} \int_0^t \| \mgenmin_s \|_{L^2}^2 ds \bigm| \cF^t\Big]
\nnb
&\leq
\E \Big[v_0^{\Lv_\epsilon} \big (Y_t^\epsilon + \Phi_t^{\Lv_\epsilon} +  I_t^\epsilon(\genmin^\epsilon)\big) + \frac{1}{2} \int_0^t \| \genmin_s^\epsilon\|_{L^2}^2 ds \bigm| \cF^t\Big],
\end{align}
and thus, since $\genmin^\epsilon$ is a minimiser of \eqref{eq:conditional-boue-dupuis},
the last display holds with equality.
It follows that
\begin{equation}
\E \Big[ \int_0^t \| \mgenmin_s\|_{L^2}^2 ds \bigm| \cF^t\Big]
= \E \Big[ \int_0^t \| \genmin^{\epsilon}_s\|_{L^2}^2 ds \bigm| \cF^t\Big] ,
\end{equation}
and thus,
\begin{equation}
\E \Big[ \int_0^t \mathbf{1}_{\genmin_s^\epsilon > 0} \| \genmin_s^{\epsilon} \|_{L^2}^2 ds \bigm| \cF^t\Big] = 0.
\end{equation}
We deduce that, conditional on $\cF^t$, we have
\begin{equation}
\int_0^t \mathbf{1}_{\genmin_s^\epsilon > 0} \| \genmin_s^{\epsilon} \|_{L^2}^2 ds = 0,
\end{equation}
which shows that $\genmin_s^\epsilon \leq 0$ for Lebesgue-a.e.\ $s\in [0,t]$.
The estimate \eqref{eq:lv-int-drivt-genmin-negative} then follows from the fact that $q_t^\epsilon$ has positive entries.
\end{proof}

Using the previous result, we can now give the proof of Proposition \ref{prop:lvshg-all-scales-drift}.

\begin{proof}[Proof of Proposition \ref{prop:lvshg-all-scales-drift} for $\cE = \Lv$]
By comparision of $\genmin^\epsilon$ with the trivial drift $u=0$ and using $v_0^{\cE_\epsilon} \geq 0$, we have
\begin{equation}
\E\big[ \int_0^t \| \genmin_s^\epsilon \|_{L^2}^2 ds \bigm | \cF^t \big] \leq 
\E\big[ v_0^{\Lv_\epsilon}(Y_t^\epsilon + \Phi_t^{\Lv_\epsilon}) \bigm| \cF^t \big]  .
\end{equation}
Since $v_0^{\Lv_\epsilon}$ is increasing in the field variables and moreover $Y_t^\epsilon + \Phi_t^{\Lv_\epsilon} \leq Y_\infty^\epsilon$ by Proposition \ref{prop:lvshg-coupling-eps} for $\cE= \Lv$,
it follows that
\begin{equation}
\E\big[ v_0^{\Lv_\epsilon} (Y_t^\epsilon + \Phi_t^{\cE_\epsilon}\big) \big] \leq \E[v_0^{\Lv_\epsilon} (Y_\infty^\epsilon ) \big]  < \infty,
\end{equation}
and the last right hand side is bounded uniformly in $\epsilon>0$.
\end{proof}

Before entering in the proof of Proposition \ref{prop:lvshg-all-scales-drift} for $\cE = \Lv$, we put forward the following result on the regularity of the Wick ordered exponential of the Gaussian free field.
The proof of this result is given in Appendix \ref{app:wickexp-regularity}.

\begin{lemma}
\label{lem:sobolev-norms-liouville-finite}
For $\delta \in (0, 1- \beta /4 \pi)$, we have
\begin{equation}
\label{eq:sobolev-norms-liouville-finite}
\sup_{\epsilon>0} \E\big[  \| \wick{\exp\big(\sqrt{\beta}Y_\infty\big)}_\epsilon \|_{H^{-1+\delta}( \Omega_\epsilon)}^2 \big] < \infty
\end{equation}
\end{lemma}

We now give the proof of the small scale estimate on minimisers $u^\epsilon$ in the case $\cE = \Lv$.

\begin{proof}[Proof of Proposition \ref{prop:lvshg-small-scales-drift} for $\cE = \Lv$]
We first note that
\begin{equation}
\label{eq:lv-estimate-for-lvdrift}
\E \Big[ \int_0^t \|\genmin_s^\epsilon\|_{L^2(\Omega_\epsilon)}^2 ds \bigm | \cF^t \Big]
\leq \sqrt{\beta} \E\big[ - \int_{\Omega_\epsilon} I_t^\epsilon(\genmin^\epsilon) \wick{\exp( \sqrt{\beta} Y_\infty^\epsilon )}_\epsilon dx \bigm | \cF^t \big].
\qquad \text{a.s.}
\end{equation}
Indeed, with the notation introduced in \eqref{eq:lvshg-coupling-eps} and \eqref{eq:Y-definition} we have
\begin{equation}
Y_t^\epsilon + \Phi_t^{\Lv_\epsilon} = Y_\infty^\epsilon + \Phi_t^{\Delta_\epsilon},
\end{equation}
where $Y_\infty^\epsilon$ is the Gaussian free field.
Since $I_t^\epsilon(\genmin^\epsilon)\leq 0$ conditional on $\cF^t$ by Lemma \ref{eq:lv-int-genmin-negative} and $\Phi_t^{\Delta_\epsilon} \leq 0$ by Proposition \ref{prop:lvshg-coupling-eps} for $\cE= \Lv$,
the estimate \eqref{eq:lv-estimate-for-lvdrift} follows from \eqref{eq:lvmin-small-scales-pre-estimate} and the monotonicity of the exponential function.

By duality, we have from \eqref{eq:lv-estimate-for-lvdrift} for any $\delta \in (0,1- \beta/4\pi)$
\begin{align}
\label{eq:lv-duality-estimate-for-small-scales}
0 \leq \E\Big[ - \int_{\Omega_\epsilon} I_t^\epsilon(\genmin^\epsilon) &\wick{\exp (\sqrt{\beta}  Y_\infty^\epsilon ) }_\epsilon dx \bigm | \cF^t \Big]
\nnb
&\leq \E\big[ \|I_t^\epsilon(\genmin^\epsilon)\|_{H^{1-\delta}(\Omega_\epsilon)} \|\wick{\exp \big(\sqrt{\beta}Y_\infty ) \big) }_\epsilon \|_{H^{-1 +\delta}(\Omega_\epsilon) }  \mid \cF^t \big] .
\end{align}
Applying the Cauchy-Schwarz inequality with respect to the conditional expectation to the last right hand side and using \eqref{eq:small-scales-integrated-drift-h-norm} gives \eqref{eq:lvshg-small-scales-drift} with 
\begin{equation}
\cW_{\delta,t}^{\Lv_\epsilon}
= \E\Big[ \|\wick{\exp\big(\sqrt{\beta}Y_\infty^\epsilon\big)}_\epsilon \|_{H^{-1+\delta}(\Omega_\epsilon)}^2 \bigm| \cF^t \Big]^{1/2}.
\end{equation}
By Jensen's inequality, we then have
\begin{equation}
\E \big[ \cW_{\delta,t}^{\Lv_\epsilon}  \big] \leq \E\Big[ \|\wick{\exp\big(\sqrt{\beta}Y_\infty^\epsilon\big)}_\epsilon \|_{H^{-1+\delta}(\Omega_\epsilon)}^2 \Big]^{1/2}.
\end{equation}
which is independent of $t>0$ and uniform in $\epsilon>0$ by Lemma \ref{lem:sobolev-norms-liouville-finite}.
\end{proof}

\subsection{sinh-Gordon}
\label{ssec:shg-small-scales-drift}

It is tempting to use the decomposition
\begin{equation}
Y_t^\epsilon + \Phi_t^{\ShG_\epsilon} = Y_\infty^\epsilon + I_{t,\infty}^\epsilon(\shgmin^\epsilon),
\end{equation}
where $Y_\infty^\epsilon$ is the full Gaussian free field and $\shgmin^\epsilon$ is as in \eqref{eq:lvshg-minimiser},
and proceed as for $\cE= \Lv$.
However, for $\cE= \ShG$, the field $I_{t,\infty}^\epsilon(\shgmin^\epsilon)$ has no definite sign.
Instead, we use that law of $\Phi_t^{\ShG_\epsilon}$ is centred for all $t\geq 0$.
We use this observation together with an argument based on the Brascamp-Lieb inequailty to further estimate the right hand side in \eqref{eq:shgmin-small-scales-pre-estimate}.
The main observation here is carry out the analysis directly with the field $Y_t^\epsilon + \Phi_t^{\ShG_\epsilon}$ .
For this to work, we need the following result, which allows to apply the Brascamp-Lieb inequality to exponential moments of $Y_t^\epsilon + \Phi_t^{\ShG_\epsilon}$

\begin{proposition}
\label{prop:density-tggf-tshg}
For $\epsilon>0$ and $t\geq 0$ let $\nu^{\ShG_\epsilon,t}$ be the law of $Y_t^\epsilon + \Phi_t^{\ShG_\epsilon}$.
Then
\begin{equation}
\label{eq:density-tggf-tshg}
d \nu^{\ShG_\epsilon,t}(\phi) \propto e^{-H_t^\epsilon(\phi)} d \phi
\end{equation}
with $H_t^\epsilon(\phi) = H_t^\epsilon(-\phi)$ and $\He H_t^\epsilon \geq (-\Delta^\epsilon + m^2)$.
\end{proposition}

\begin{proof}
To ease the notation we drop $\epsilon>0$ from the notation.
Let $\nu_1$ be the law of $Y_t$ and let $\nu_2$ be the law of $\Phi_t^{\ShG}$.
Note that we have
\begin{equation}
d \nu_1(\varphi) \propto e^{- \frac{1}{2}\varphi c_t^{-1} \varphi} d\varphi,
\qquad \text{and} \qquad
d \nu_2(\varphi) \propto e^{-v_t^\ShG(\varphi) - \frac{1}{2} \varphi (c_\infty - c_t)^{-1} \varphi} d\varphi.
\end{equation}
Since $Y_t$ and $\Phi_t^\ShG$ are independent, we have $\nu^{\ShG,t} = \nu_1*\nu_2$, and thus, 
\begin{align}
d \nu (\varphi)
&\propto \Big( 
\int e^{- \frac{1}{2}(\zeta-\varphi) c_t^{-1} (\zeta-\varphi)} 
e^{-v_t^\ShG(\zeta) - \frac{1}{2} \zeta (c_\infty - c_t)^{-1} \zeta} d \zeta
\Big) d\varphi
\nnb
&= e^{-\frac{1}{2} \varphi c_t^{-1} \varphi}
\Big(
\int e^{ -\frac{1}{2}\zeta \big( c_t^{-1} + (c_\infty - c_t)^{-1}  \big) \zeta + \varphi c_t^{-1} \zeta - v_t^\ShG(\zeta) } d\zeta 
\Big)
= e^{-\big( \frac{1}{2} \varphi c_t^{-1} \varphi + U_t(\varphi) \big)} d\varphi,
\end{align}
where we set
\begin{equation}
\label{eq:u-t-definition}
U_t(\varphi) = - \log \Big( \int e^{ -\frac{1}{2}\zeta \big( c_t^{-1} + (c_\infty - c_t)^{-1}  \big) \zeta + \varphi c_t^{-1} \zeta - v_t(\zeta) } d\zeta
\Big) .
\end{equation}
Since $v_t^\ShG(\zeta) = v_t^\ShG(-\zeta)$, we have that $U_t(\varphi) = U_t(-\varphi)$, and thus, the law $\nu^{\ShG,t}$ is centred.

To prove the lower bound on $\He H_t$, we need to find a lower bound on $\He U_t$.
Differentiation of \eqref{eq:u-t-definition} gives
\begin{equation}
\nabla U_t(\varphi)
= - c_t^{-1}
\frac{\int \zeta e^{ -\frac{1}{2}\zeta \big( c_t^{-1} + (c_\infty - c_t)^{-1}  \big) \zeta + \varphi c_t^{-1} \zeta - v_t^\ShG(\zeta) } d\zeta}{\int e^{ -\frac{1}{2}\zeta \big( c_t^{-1} + (c_\infty - c_t)^{-1}  \big) \zeta + \varphi c_t^{-1} \zeta - v_t^\ShG(\zeta) } d\zeta} ,
\end{equation}
and further
\begin{equation}
\He U_t (\varphi) = - c_t^{-1} \cov_{\mu^\varphi}(\zeta) c_t^{-1} ,
\end{equation}
where the measure $d\mu^\varphi$ is defined by
\begin{equation}
d \mu^\varphi(\zeta) \propto e^{-\frac{1}{2}\zeta \big( c_t^{-1} + (c_\infty - c_t)^{-1}  \big) \zeta + \varphi c_t^{-1} \zeta - v_t^\ShG(\zeta)} d\zeta.
\end{equation}
Thus, we have
\begin{equation}
\He H_t (\varphi) = c_t^{-1} + \He U_t = 
c_t^{-1} - c_t^{-1} \cov_{\mu^\varphi}(\zeta) c_t^{-1}.
\end{equation}
To find a lower bound on $\He H_t$, we need to find an upper bound on $\cov_{\mu^\varphi}(\zeta)$.
To this end, we first define
\begin{equation}
W_t^\varphi(\zeta) = \frac{1}{2}\zeta \big( c_t^{-1} + (c_\infty - c_t)^{-1}  \big) \zeta - \varphi c_t^{-1} \zeta + v_t^\ShG(\zeta).
\end{equation}
Since $\He v_t^{\ShG} \geq 0$ by Lemma \ref{lem:lvshg-he-v-pos-definite}, we have
\begin{equation}
\He W_t^\varphi (\zeta) = c_t^{-1} + (c_\infty - c_t)^{-1} + \He v_t^\ShG(\varphi) \geq c_t^{-1} + (c_\infty - c_t)^{-1}.
\end{equation}
Thus, we have, by the Brascamp-Lieb inequality \cite[Theorem 4.1]{MR0450480},
\begin{equation}
\cov_{\mu^\varphi} (\zeta) \leq \big(c_t^{-1} + (c_\infty - c_t)^{-1} \big)^{-1} .
\end{equation}
Therefore, we have
\begin{equation}
\label{eq:he-h-first-lower-bound}
\He H_t (\varphi) \geq 
c_t^{-1} - c_t^{-1} \big(c_t^{-1} + (c_\infty - c_t)^{-1} \big)^{-1}c_t^{-1}.
\end{equation}
We now show that the last right hand side is equal to $(-\Delta +m^2)$.
To this end, we first observe
\begin{equation}
c_\infty - c_t = (- \Delta + m^2)^{-1} - (-\Delta + m^2 + 1/t)^{-1} = \frac{1}{t} (- \Delta + m^2)^{-1} (- \Delta + m^2+ 1/t)^{-1},
\end{equation}
and thus,
\begin{equation}
(c_\infty - c_t)^{-1} = t c_\infty^{-1} c_t^{-1}.
\end{equation}
Then, we have from \eqref{eq:he-h-first-lower-bound}
\begin{align}
\He H_t (\varphi) &\geq c_t^{-1} - c_t^{-1}\big( c_t^{-1} + t c_\infty^{-1} c_t^{-1}\big)^{-1}c_t^{-1}
= 
c_t^{-1} - \big( \id + t c_\infty^{-1}\big)^{-1}c_t^{-1}
\nnb
&= c_t^{-1} \big( \id - (\id + tc_\infty^{-1})^{-1} \big) 
= c_t^{-1} (\id +tc_\infty^{-1})^{-1} \big( \id+tc_\infty^{-1} - \id \big)
\nnb
&= 
t c_t^{-1} (\id + tc_\infty^{-1})^{-1} c_\infty^{-1} = c_\infty^{-1} =
(-\Delta + m^2),
\end{align}
where we used that
\begin{equation}
(\id + tc_\infty^{-1}) = t( -\Delta + m^2 + 1/t) = t c_t^{-1}.
\end{equation}
\end{proof}

We can now give the proof of Proposition \ref{prop:lvshg-all-scales-drift} for $\cE = \ShG$.

\begin{proof}[Proof of Proposition \ref{prop:lvshg-all-scales-drift} for $\cE = \ShG$]
As in the proof of Proposition \ref{prop:lvshg-all-scales-drift} for $\cE = \Lv$,
we compare with the trivial drift $u=0$ and use $v_0^{\ShG_\epsilon} \geq 0$ to obtain
\begin{equation}
\E\big[ \int_0^t \| \genmin_s^\epsilon \|_{L^2}^2 ds \bigm | \cF^t \big] \leq 
\E\big[ v_0^{\ShG_\epsilon}(Y_t^\epsilon + \Phi_t^{\ShG_\epsilon}) \bigm| \cF^t \big]. 
\end{equation}
Now, we use Proposition \ref{prop:density-tggf-tshg} and the Brascamp-Lieb inequality for exponential moments to obtain
\begin{equation}
\E\big[ v_0^{\ShG_\epsilon} \big(Y_t^\epsilon + \Phi_t^{\ShG_\epsilon} \big) \big]
\leq \E[v_0^{\ShG_\epsilon} (Y_\infty^\epsilon ) \big]  < \infty,
\end{equation}
where the last right hand side is bounded uniformly in $\epsilon>0$.
\end{proof}

Another consequence of Proposition \ref{prop:density-tggf-tshg} is the following statement on the regularity of the Wick ordered exponentials under the law $\nu^{\ShG_\epsilon,t}$
Its proof is a variation of the proof of Lemma \ref{lem:sobolev-norms-liouville-finite} and also presented in Appendix \ref{app:wickexp-regularity}.

\begin{lemma}
\label{lem:sobolev-norms-sinh-finite}
For $\delta \in (0, 1- \beta /4 \pi)$, we have
\begin{equation}
\label{eq:sobolev-norms-sinh-finite}
\E\big[  \| \wick{\exp\big(\pm\sqrt{\beta}(Y_t^\epsilon + \Phi_t^{\ShG_\epsilon})\big)}_\epsilon \|_{H^{-1+\delta}}^2 \big] \leq
\E\big[  \| \wick{\exp\big(\pm\sqrt{\beta}Y_\infty^\epsilon \big)}_\epsilon \|_{H^{-1+\delta}}^2 \big] .
\end{equation}
In particular,
\begin{equation}
\sup_{\epsilon \geq 0} \sup_{t\geq 0} \E\big[  \| \wick{\exp\big(\pm \sqrt{\beta} (Y_t^\epsilon + \Phi_t^{\ShG_\epsilon}) \big)}_\epsilon \|_{H^{-1+\delta}}^2 \big] < \infty .
\end{equation}
\end{lemma}

With the regularity estimates at hand, we can now give the proof of Proposition \ref{prop:lvshg-small-scales-drift} for $\cE = \ShG$.

\begin{proof}[Proof of Proposition \ref{prop:lvshg-small-scales-drift} for $\cE = \ShG$]
By duality, we have from \eqref{eq:shgmin-small-scales-pre-estimate} for any $\delta \in (0,1- \beta/4\pi)$
\begin{align}
\label{eq:shg-estimate-duality}
0 \leq \E\Big[ &- \int_{\Omega_\epsilon} I_t^\epsilon(\genmin^\epsilon) \wick{\sinh (\sqrt{\beta}  Y_t^\epsilon + \Phi_t^{\ShG_\epsilon}) }_\epsilon dx \bigm | \cF^t \Big]
\nnb
&\leq 
\E\Big[ \| I_t^\epsilon(\genmin^\epsilon)\|_{H^{1-\delta}} \Big(\|\wick{\exp \big(\sqrt{\beta}(Y_t^\epsilon + \Phi_t^{\ShG_\epsilon} ) \big) }_\epsilon \|_{H^{-1 +\delta} }
\nnb
&\qquad \qquad \qquad  \qquad \qquad + \|\wick{\exp \big(-\sqrt{\beta}(Y_t^\epsilon + \Phi_t^{\ShG_\epsilon} ) \big) }_\epsilon \|_{H^{-1 +\delta} }  \Big) \bigm| \cF^t \Big] .
\end{align}
Applying the Cauchy-Schwarz inequality with respect to the conditional expectation to the last right hand side and using \eqref{eq:small-scales-integrated-drift-h-norm} gives \eqref{eq:lvshg-small-scales-drift} for $\cE= \ShG$ with 
\begin{align}
\cW_{\delta,t}^{\ShG_\epsilon}
&= \E\Big[ \|\wick{\exp\big(\sqrt{\beta}(Y_t^\epsilon + \Phi_t^{\ShG_\epsilon})\big)}_\epsilon \|_{H^{-1+\delta}}^2 \bigm| \cF^t \Big]^{1/2}
\nnb
&\qquad  +
\E\Big[ \|\wick{\exp\big(-\sqrt{\beta}(Y_t^\epsilon + \Phi_t^{\ShG_\epsilon}) \big)}_\epsilon \|_{H^{-1+\delta}}^2 \bigm| \cF^t \Big]^{1/2}
.
\end{align}
By Jensen's inequality, we then have
\begin{align}
\E \big[ \cW_{\delta,t}^{\ShG_\epsilon} \big]
&\leq \E\Big[ \|\wick{\exp\big(\sqrt{\beta}( Y_t^\epsilon + \Phi_t^{\ShG_\epsilon})\big)}_\epsilon \|_{H^{-1+\delta}}^2 \Big]^{1/2}
\nnb
&\qquad + \E\Big[ \|\wick{\exp\big(-\sqrt{\beta}( Y_t^\epsilon + \Phi_t^{\ShG_\epsilon})\big)}_\epsilon \|_{H^{-1+\delta}}^2 \Big]^{1/2} ,
\end{align}
which is bounded uniformly in $t\geq 0$ and $\epsilon>0$ by Lemma \ref{lem:sobolev-norms-sinh-finite}.
\end{proof}

\section{Proof of the main results}
\label{sec:proof-main-results}

In this section, we give the proof of the main results, Theorem \ref{thm:coupling-pphi-to-gff-eps} and Corollary \ref{cor:pphi-coupling-continuum}.
To this end, we apply the estimates on generic minimisers to the specific minimiser $\lvshgmin^\epsilon$ in \eqref{eq:lvshg-minimiser} and recall that
\begin{equation}
\Phi_t^{\Delta_\epsilon} = I_{t, \infty}^\epsilon(\lvshgmin^\epsilon), \qquad \Phi_0^{\Delta_\epsilon} - \Phi_t^{\Delta_\epsilon} = I_t^\epsilon(\lvshgmin^\epsilon).
\end{equation}
To prove the estimates in Theorem \ref{thm:coupling-pphi-to-gff-eps}, we decompose the field $\Phi_t^{\Delta_\epsilon}$ in different scales, which motivates the notation
\begin{equation}
\Phi_{s,t}^{\Delta_\epsilon} = -\int_s^t \dot c_\tau^\epsilon \nabla v_\tau^{\cE_\epsilon}(\Phi_\tau^{\cE_\epsilon}) d\tau.
\end{equation}
The proof of Theorem \ref{thm:coupling-pphi-to-gff-eps} is now identical for $\cE = \Lv$ and $\cE= \ShG$.

\begin{proof}[Proof of Theorem \ref{thm:coupling-pphi-to-gff-eps}]
With the coupling established in Proposition \ref{prop:lvshg-coupling-eps},
it remains to prove the bounds on the difference field $\Phi^{\Delta_\epsilon}$.
We deduce these from the bounds for the short scales of minimisers in Proposition \ref{prop:lvshg-small-scales-drift}, applied to  the particular choice $\genmin^\epsilon = \lvshgmin^\epsilon$.
The bounds \eqref{eq:phi-delta-h1} and \eqref{eq:phi-delta-h2} follow from Lemma \ref{lem:sobolev-norm-integrated-drift-1} and Lemma \ref{lem:sobolev-norm-integrated-drift-2} together with Lemma \ref{prop:lvshg-all-scales-drift}.

To show the estimate \eqref{eq:phi-delta-1-2},
we first note that, for $\alpha\in [1,2-\beta/4\pi)$ and $0 < s\leq t$, we have by Lemma \ref{lem:sobolev-norm-integrated-drift-2}
\begin{equation}
\| \Phi_{s,t}^{\Delta_\epsilon} \|_{H^\alpha}
\lesssim \Big( \frac{t-s}{s^\alpha} \int_s^t \|\lvshgmin_\tau^\epsilon \|_{L^2}^2 d\tau \Big)^{1/2}.
\end{equation}
Thus, by Proposition \ref{prop:lvshg-small-scales-drift}, we have
\begin{align}
\label{eq:phi-s-t-expectation-h-alpha}
\E \big[ \| \Phi_{s,t}^{\Delta_\epsilon}\|_{H^\alpha} \big]
\lesssim
\Big(\frac{t-s}{s^\alpha} \Big)^{1/2} \E\Big[  \E\Big[ \Big( \int_s^t \| \lvshgmin_\tau^\epsilon\|_{L^2}^2 d\tau \Big)^{1/2} \bigm| \cF^t \Big] \Big]
\lesssim
\Big(\frac{t-s}{s^\alpha} \Big)^{1/2} t^{\delta/2}.
\end{align}
Let $(t_n)_{\N_0}$ be a sequence with $t_n \to 0$ as $n\to \infty$.
Then, the triangle inequality together with \eqref{eq:phi-s-t-expectation-h-alpha} and \eqref{eq:phi-delta-h2} gives, for any $\delta \in (0,1-\beta/4\pi)$,
\begin{align}
\label{eq:phi-0-delta-exp-sobolev-norm-upper-by-triangle}
\E\big[ \| \Phi_0^{\Delta_\epsilon} \|_{H^{\alpha}} \big]
&\leq 
\sum_{n\in \N} \E\big[ \| \Phi_{t_{n+1},t_n}^{\Delta_\epsilon} \|_{H^\alpha} \big]
+ \E\big[ \| \Phi_{t_0}^{\Delta_\epsilon}\|_{H^\alpha} \big]
\nnb
&\leq \sum_{n\in \N} \Big( \frac{t_n - t_{n+1}}{t_{n+1}^\alpha} \Big)^{1/2} t_n^{\delta/2} + C . 
\end{align}
With the choices $t_n=2^{-n}$ and $\delta\in (\alpha-1, 1-\beta/4\pi)$,
we have
\begin{equation}
\Big( \frac{t_n - t_{n+1}}{t_{n+1}^\alpha} \Big)^{1/2} t_n^{\delta/2}
= 
\big(2^{-(n+1) +n\alpha} 2^{-\delta n} \big)^{1/2} = 2^{-1/2}2^{-\frac{1}{2} n(1-\alpha +\delta)} ,
\end{equation}
and thus, the sum on the right hand side of \eqref{eq:phi-0-delta-exp-sobolev-norm-upper-by-triangle} is finite.

The convergence \eqref{eq:phi-delta-to-0} is obtained from a similar reasoning reasoning.
We first note that
\begin{equation}
\Phi_0^{\Delta_\epsilon} - \Phi_t^{\Delta_\epsilon} = \int_0^t \dot c_s^\epsilon \nabla v_t^{\cE_\epsilon}(\Phi_s^{\cE_\epsilon}) ds .
\end{equation}
For $t>0$ let $(\tau_n)_{n\in \N_0}$ be defined by $\tau_n = t2^{-n}$ and choose $\delta \in (\alpha-1, 1- \beta/4\pi)$.
Then the same estimates that lead to \eqref{eq:phi-0-delta-exp-sobolev-norm-upper-by-triangle} now give
\begin{align}
\E\big[ \| \Phi_0^{\Delta_\epsilon} - \Phi_t^{\Delta_\epsilon} \|_{H^\alpha} \big]
&\lesssim \sum_{n\in \N} \Big( \frac{\tau_n - \tau_{n+1}}{\tau_{n+1}^\alpha} \Big)^{1/2} \tau_n^{\delta/2}
= t^{\frac{1}{2}(1-\alpha + \delta) } \sum_{n\in \N} 2^{-1/2}2^{-\frac{1}{2} n(1-\alpha +\delta)} \lesssim t^{\frac{1}{2}(1-\alpha + \delta)} ,
\end{align}
which shows the claimed convergence as $t\to 0$.
\end{proof}


The next results is the convergence of the process $\Phi^{\Delta_\epsilon}$ as $\epsilon \to 0$ to a process $\Phi^{\Delta_0}$ along a suitable subsequence.
We deduce this convergence from the tightness of the sequence $(\Phi^{\Delta_\epsilon})_\epsilon$, which follows form the Arz\`ela-Ascoli theorem.
This is completely analogous to \cite[Proposition 5.5]{MR4665719}, and we include it here for completeness.
To have all lattice fields taking values in the same space, we use the isometric embedding $I_\epsilon \colon L^2(\Omega_\epsilon) \to L^2(\Omega)$,
which is obtained from extending the Fourier series of a given function $f\in X_\epsilon$ to $k\in \Omega^* = 2\pi \Z^2$,
i.e., for $\Phi^\epsilon \in X_\epsilon$ with Fourier series
\begin{equation}
\Phi^\epsilon (x) = \sum_{k\in \Omega_\epsilon^*} \hat r^\epsilon(k) e^{ikx} ,
\end{equation}
we have that $I_\epsilon \Phi^\epsilon$ has Fourier coefficients $\hat r^\epsilon(k)$ for $k\in \Omega_\epsilon^*$ and vanishing Fourier coefficients for $k\in \Omega^*\setminus \Omega_\epsilon^*$.

\begin{proposition}
\label{prop:weak-convergence-phi-delta-eps}
Let $\alpha<1$. Then $(I_\epsilon \Phi^{\Delta_\epsilon})_\epsilon$ is a tight sequence of processes in $C_0([0,\infty), H^{\alpha}(\Omega))$.
In particular, there is a process $\Phi^{\Delta_0} \in C_0([0,\infty), H^{\alpha}(\Omega))$ and a subsequence $(\epsilon_k)_k$, $\epsilon_k \to 0$ as $k\to\infty$
such that the laws of $\Phi^{\Delta_{\epsilon_k}}$ on $C_0([0,\infty), H^{\alpha}(\Omega))$ converge weakly to the law of $\Phi^{\Delta_0}$.
\end{proposition}

\begin{proof}
For $R>0$ and $\alpha<1$, let
\begin{align}
\cX_R
=
\Big \{ \Phi \in C_0([0,\infty),H^\alpha(\Omega)) \colon \sup_{t \in [0,\infty)} \|\Phi\|_{H^\alpha}^2 \leq R \text{ and } \sup_{s< t} \frac{\|\Phi_t - \Phi_s\|_{H^\alpha}^2}{(t-s)^{1-\alpha}} \leq R \Big\}.	
\end{align}
The set $\cX_R$ is bounded and equicontinuous with respect to the norm $\|\cdot \|_{H^\alpha}$, and thus, by the Arz\`ela-Ascoli theorem, the closure $\cX_R$ is compact.
Moreover, we have
\begin{align}
\sup_{t\geq 0}\|I_\epsilon \Phi_t^{\Delta_\epsilon} \|_{H^\alpha}^2 & \lesssim \int_0^\infty \|\lvshgmin_\tau^\epsilon\|_{L^2}^2 d\tau ,
\\
\|I_\epsilon \Phi_t^{\Delta_\epsilon} - I_\epsilon \Phi_s^{\Delta_\epsilon} \|_{H^\alpha}^2
&\lesssim (t-s)^{1-\alpha} \int_0^\infty \| \lvshgmin_\tau^\epsilon\|_{L^2}^2 d\tau ,
\end{align}
and thus, we have as in the proof of \cite[Theorem 1.1]{MR4665719} for some constant $C>0$
\begin{align}
\P(I_\epsilon\Phi^{\Delta_\epsilon} \in \overline{\cX_R}^c )
&\leq 
\P ( I_\epsilon \Phi^{\Delta_\epsilon} \in \cX_R^c )
\leq 
\P \Big( \sup_{t \in [0,\infty)} \|\Phi_t^{\Delta_\epsilon} \|_{H^\alpha}^2 +
 \sup_{s< t} \frac{\|\Phi_t^{\Delta_\epsilon} - \Phi_s^{\Delta_\epsilon} \|_{H^\alpha}^2}{(t-s)^{1-\alpha}} > R \Big)
\nnb
&\leq \P \Big( \int_0^\infty \|\lvshgmin_\tau^\epsilon\|_{L^2}^2 d\tau > R/C \Big) \leq
\frac{C}{R} \E \Big[ \int_0^\infty \|\lvshgmin_\tau^\epsilon \|_{L^2}^2 d\tau \Big].
\end{align}
By Proposition \ref{prop:lvshg-all-scales-drift}, for a given $\kappa>0$, we choose $R$ large enough such that
\begin{equation}
\sup_{\epsilon>0} \P ( I_\epsilon \Phi^{\Delta_\epsilon} \in \overline{\cX_R}^c ) \leq \frac{2}{R} \sup_{\epsilon>0} \E \Big[ \int_0^\infty \|\lvshgmin_\tau\|_{L^2}^2 d\tau \Big] <\kappa,
\end{equation}
which establishes tightness for the sequence $(I_\epsilon\Phi^{\Delta_\epsilon})_\epsilon \subseteq C_0([0,\infty), H^{\alpha}(\Omega))$. The existence of a weak limit $\Phi^{\Delta_0}$ then follows by Prohorov's theorem.
\end{proof}

Next, we discuss the convergence of the marginals of the process $(\Phi^{\cE_\epsilon})_{\epsilon}$ for $t\geq 0$,
thereby showing that every weak limit obtained from Proposition \ref{prop:weak-convergence-phi-delta-eps} has the same law at least for a fixed $t\geq 0$.
Before, we state state and prove the following preliminary convergence result for the Gaussian multiplicative chaos.
To this end, we set
\begin{equation}
M^{\epsilon}(\phi) = \int_{\Omega_\epsilon} e^{\sqrt{\beta}\phi_x - \frac{\beta}{2}c_\infty(x,x) } dx
\end{equation}
as well as
\begin{equation}
\bar M(\phi) 
= \frac{1}{2} \big( M^{\epsilon,+}(\phi) + M^{\epsilon,-} (\phi)\big)
= \frac{1}{2} \int_{\Omega_\epsilon} \big(e^{\sqrt{\beta}\phi_x - \frac{\beta}{2} c_\infty(x,x)} + e^{-\sqrt{\beta}\phi_x - \frac{\beta}{2} c_\infty (x,x)} \big) dx .
\end{equation}
Below, we consider these objects when $\phi= Y_\infty^\epsilon$, in which case we write $M^\epsilon(Y_\infty^\epsilon) \equiv M^\epsilon$ and 
$\bar M^\epsilon(Y_\infty^\epsilon) \equiv \bar M^\epsilon$,
where we recall that all random fields $(Y_\infty^\epsilon)_\epsilon$ are realised on the same probability space. 
The following result gives convergence in $L^1$ as $\epsilon \to 0$ to limiting random variables $M$ and $\bar M$.
We state it here for $\beta \in (0,4\pi)$ and make the remark that the same holds true also for $\beta\in (0,8\pi)$.
For the later use, we also define
\begin{equation}
e^{-v_\infty^{\Lv_0}(0)} = \E\big[ e^{- \lambda M }\big],
\qquad
e^{-v_\infty^{\ShG_0}(0)} = \E\big[ e^{- \lambda \bar M }\big],
\end{equation}
where $\lambda>0$ is as in \eqref{eq:lv-density} and \eqref{eq:shg-density}.

\begin{lemma}
\label{lem:wick-exponential-convergence}
Let $\beta \in (0,4\pi)$. There exist non-negative random variables $M$ and $\bar M$,
such that as $\epsilon \to 0$, we have
\begin{align}
\label{eq:wick-exponential-convergence}
M^{\epsilon} \to M \qquad \text{and} \qquad \bar M^{\epsilon} \to \bar M \qquad \text{in $L^1$.}
\end{align}
\end{lemma}

\begin{proof}
We deduce the convergence from \cite[Theorem 25]{MR3475456}.
To this end, we need to verify three conditions.
To verify the first condition we show that $(M^\epsilon)_\epsilon$ is bounded in $L^2$, which implies uniform integrability.
To this end, we note that
\begin{equation}
 \E\big[(M^\epsilon \big)^2\big] \lesssim \int_{\Omega_\epsilon\times \Omega_\epsilon} |x-y|^{- \beta/2\pi} dx dy ,
\end{equation}
and the last right hand side is bounded uniformly in $\epsilon>0$ for $\beta \in (0,4\pi)$.
The second and the third condition are verified in the proof of \cite[Lemma 2.7]{HofstetterZeitouni2025Liouville}.
\end{proof}

\begin{proposition}
\label{prop:weak-convergence-nu-t}
As $\epsilon \to 0$, we have for $t>0$ and as measures on $H^\alpha(\Omega)$ for every $\alpha <1$,
that
$(I_{\epsilon})_*\nu_t^{\cE_\epsilon}$ 
converges weakly to $\nu_t^\cE$ given by
\begin{align}
\label{eq:nu-t-continuum}
\E_{\nu_{t}^\cE}[F]
=
e^{v_\infty^{\cE_0}(0)}  \E \big[ F(Y_\infty - Y_t) e^{-v_t^\cE(Y_\infty-Y_t)} \big]
\end{align}
for $F\colon H^\alpha (\Omega) \to \R$ bounded and measurable.

Moreover, for $t=0$, the weak convergence $(I_{\epsilon})_*\nu_0^{\cE_\epsilon} \to \nu_0^\cE$ holds as measures on $H^\alpha(\Omega)$ for any $\alpha <0$ and with $\nu_0^\cE$ defined by \eqref{eq:nu-t-continuum} with $t=0$ and $F\colon H^{\alpha}(\Omega) \to \R$ for $\alpha <0$.
\end{proposition}

\begin{proof}
We first recall from \eqref{eq:lvshg-renormalised-measure} that the renormalised measure $\nu_t^{\cE_\epsilon}$ is defined by
\begin{align}
\E_{\nu_t^{\cE_\epsilon}}[F]
=
e^{v_\infty^{\cE_\epsilon}(0)}  \EE_{c_\infty^\epsilon - c_t^\epsilon} [F(\zeta) e^{-v_t^{\cE_\epsilon}(\zeta)}  ] 
=
{ e^{v_\infty^{\cE_\epsilon}(0)} } \E[F(Y^{\epsilon}_{\infty}-Y^{\epsilon}_{t})   e^{-v^{\cE_\epsilon}_{t}(Y^{\epsilon}_{\infty}-Y^{\epsilon}_{t})}],
\end{align}
where $F\colon X_\epsilon \to \R$ is bounded and continuous,
and $v_t^{\cE_\epsilon}$ is the renormalised potential.
By the definition of the pushforward measure and the renormalised potential we obtain that,
for $F\colon H^\alpha(\Omega)\to\R$ bounded and continuous,
\begin{align}
\E_{(I_{\epsilon})_*\nu^{\cE_\epsilon}_{t}}[F]
&= \E_{\nu_t^{\cE_\epsilon}} [F \circ I_\epsilon] 
=  e^{v_\infty^{\cE_\epsilon}(0)} \E [F \big( I_\epsilon (Y_\infty^\epsilon - Y_t^\epsilon) \big)  e^{-v_t^{\cE_\epsilon}(Y_\infty^\epsilon - Y_t^\epsilon)} ] =  
\nnb
&= { e^{v_\infty^{\cE_\epsilon}(0)} } \E\Big[ F\big(I_{\epsilon}(Y_\infty^\epsilon -Y_t^\epsilon)\big) 
\E \big[e^{- v_0^{\cE_\epsilon}(Y_\infty^\epsilon - Y_t^\epsilon + Y_t^\epsilon) } \bigm| \cF^t \big] \Big]
\nnb
&= e^{v_\infty^{\cE_\epsilon}(0)}  \E\Big[F\big(I_\epsilon(Y_\infty^\epsilon - Y_t^\epsilon)\big) e^{-v_0^{\cE_\epsilon}(Y_\infty^\epsilon)} \Big].
\end{align}
For $t>0$, we have by \cite[Lemma 5.4]{MR4665719} that $I_{\epsilon}(Y^{\epsilon}_{\infty}-Y^{\epsilon}_{t})$ converges to $Y_{\infty}-Y_{t}$ in $L^2$ with respect to the norm of $H^\alpha(\Omega)$ for any $\alpha <1$.
Moreover, we have by Lemma \ref{lem:wick-exponential-convergence} that $v_0^{\cE_\epsilon} (Y_\infty^\epsilon) \to v_0^{\cE_0} (Y_\infty)$ in $L^1$.
Since $v_0^{\cE_\epsilon} \geq 0$, it follows that
\begin{equation}
\label{eq:vitali-convergence-l1}
\E \Big [F\big(I_\epsilon(Y_\infty^\epsilon - Y_t^\epsilon) \big) e^{-v_0^{\cE_\epsilon} (Y_\infty^{\epsilon})} \Big]
\to
\E\Big[F(Y_\infty^\epsilon - Y_t^\epsilon) e^{-v_0^{\cE_0} (Y_\infty^\epsilon)} \Big],
\end{equation}
from which the claimed convergence follows.

For the case $t=0$ we follow the same arguments as for $t>0$, but now we take $F\colon H^\alpha (\Omega) \to \R$ for $\alpha <0$ and use \cite[(5.4)]{MR4665719}. 
\end{proof}

\begin{proof}[Proof of Corollary \ref{cor:pphi-coupling-continuum}]
Since $\Phi^{\Delta_{\epsilon_k}} \to \Phi^{\Delta_{0}}$ in distribution as $k\to \infty$,
we also have that there exists a process $\Phi_t^{\cE_0} \equiv \Phi^{\Delta_0} + \Phi^{\GFF_0}$,
such that $\Phi^{\cE_{\epsilon_k}} \to \Phi^{\cE_{0}}$ in distribution as $k\to \infty$. 
Moreover, as $\epsilon\to 0$,
we have, for any $t\geq 0$, that $\nu_t^{\cE_\epsilon} \to \nu_t^{\cE}$ where $\nu_t^{\cE}$ denotes the distribution in \eqref{eq:nu-t-continuum}.
In particular, $(\Phi_0^{\cE_\epsilon} )_\epsilon$ converges in distribution to the continuum Liouville measure (for $\cE = \Lv$) and sinh-Gordon measure (for $\cE= \ShG$).

Finally, the estimates on the norms of $ \Phi^{\Delta_0}$ and the independence of $\Phi_t^{\cE_0}$ and $\Phi_0^{\GFF_0} - \Phi_t^{\GFF_0}$ follow from the convergence in distribution  and the uniform bounds on the level of regularisations in Theorem \ref{thm:coupling-pphi-to-gff-eps}.
\end{proof}

\appendix

\section{Existence of a solutions to the Polchinski SDE for $\epsilon>0$}
\label{app:proof-sde-existence}

In this section we show that the SDE \eqref{eq:lvshg-coupling-eps} has a unique solution  and that the marginals for a fixed $t\geq 0$ are distributed as the renormalised measure in \eqref{eq:lvshg-renormalised-measure}.
Since $\epsilon >0$ is fixed, we drop it from the notation througout the entire section.
Note that, with the convention to be interpreted as a backward SDE, the differential form of \eqref{eq:lvshg-coupling-eps} reads
\begin{equation}
\label{eq:lvshg-sde-differential-form}
d\Phi_t = - \dot c_t \nabla v_t^\cE(\Phi_t) + \dot c_t^{1/2} dW_t, \qquad \Phi_\infty = 0. 
\end{equation}

\begin{proof}[Proof of Proposition \ref{prop:lvshg-coupling-eps}]

For $\cE = \Lv$ the existence and uniqueness of a solution to  \eqref{eq:lvshg-sde-differential-form} was established in \cite[Theorem 3.1]{HofstetterZeitouni2025Liouville} for a different covariance regularisation even for $\epsilon=0$.
We can follow the same argument, which is based on Picard iterations, to obtain this result also for the choice $(c_t)_{t\in [0,\infty]}$ as in \eqref{eq:ct-pauli-villars}.
Note that, for this choice, $c_t$ has positive entries and thus, since $\dot c_t = c_t^2 t^{-2}$,
the same holds for $\dot c_t$.
In particular, the $\dot c_t$ is positivity preserving as an operator $X_\epsilon \to X_\epsilon$, from which the determisistic sign of $\Phi^{\Delta_\epsilon}$ follows.

Thus, it remains to discuss the case $\cE= \ShG$.
To prove existence of a solution up to $t=0$,
it is convenient to consider the square of the $L^2$ norm of the process $q_t \Phi_t^{\cE}$ where 
\begin{equation}
\label{eq:sob-norm-t}
q_t = (\dot c_t)^{-1/2} = t (c_t)^{-1}
= t (-\Delta+ m^2) + 1 .
\end{equation}
For clarity we further note that
\begin{equation}
\|q_t \Phi_t^\ShG  \|_{L^2}^2 = \Phi_t^\ShG \dot c_t^{-1} \Phi_t^\ShG.
\end{equation}

We first prove that the SDE \eqref{eq:lvshg-sde-differential-form} has a solution up to $t=0$.
To this end, we define, for $n\in \N_0$, random variables $T_n \in [0,\infty]$ by
\begin{equation}
T_n = \sup \{ t\geq 0 \colon g_t \|q_t \Phi_t^\ShG \|_{L^2}^2 \geq n \}
\end{equation}
for a positive and differentiable function $g\colon [0,\infty) \mapsto \R$ with $g_t \to 0$ as $t\to \infty$ to be determined below.
Furthermore, we define $T_\infty = \lim_{n\to \infty} T_n$, which is well-defined, since $(T_n)_n$ is decreasing as $n\to \infty$.
Note that $(T_n)_n$ and $T_\infty$ are stopping times with respect to the backward filtration $\cF^t$.

In what follows, we show that $T_\infty = 0$ a.s., which implies that a global to \eqref{eq:lvshg-sde-differential-form} solution exists a.s.
To this end, we apply the local It\^o formula to $f(t\vee T_n,\Phi_{t\vee T_n}^\ShG)$,
where the function $f\colon [0,\infty)\times X_\epsilon \to \R$ is defined by 
\begin{equation}
f(t,\Phi) 
= g_t
\Phi (\dot c_t)^{-1} \Phi = g_t \| q_t \Phi \|_{L^2}^2.
\end{equation}
Noting that 
\begin{align}
\label{eq:f-derivatives}
\nabla f = 2 g_t \dot c_t^{-1} \Phi ,
\qquad
\frac{\partial f}{\partial t}
=  - g_t \Phi \ddot c_t (\dot c_t)^{-2} \Phi + g_t'\Phi  (\dot c_t)^{-1} \Phi , 
\qquad
\He f = 2 g_t \dot c_t^{-1},
\end{align}
it follows that, using also $g_\infty = 0$ and $\Phi_\infty =0$,
\begin{align} 
\label{eq:ito-ftphi-integrated}
f(t\vee T_n, \Phi_{t\vee T_n}^\ShG)
&= - 2  \int_{t\vee T_n}^\infty  g_s \dot c_s^{-1} \Phi_s^\ShG   \dot c_s \nabla v_s^\ShG(\Phi_s^\ShG)  ds 
+ \int_{t\vee T_n}^\infty g_s \Phi_s^\ShG \ddot c_s \dot c_s^{-2} \Phi_s^\ShG ds
\nnb
& \qquad - \int_{t\vee T_n}^\infty g_s' \| q_s \Phi_s^\ShG \|_{L^2}^2 \, ds
+ \int_{t\vee T_n}^\infty \tr(\id) g_s \, ds
\nnb
&\qquad + 2 \int_{t\vee T_n}^\infty g_s (\dot c_s^{-1} \Phi_s^\ShG) \cdot \dot c_s^{1/2} d W_s,
\end{align}
where $\tr(\id)$ denotes the trace of the identity operator on $X_\epsilon$ and the signs of the second and third integral on the right hand side of the previous display are reversed compared to \eqref{eq:f-derivatives},
since \eqref{eq:lvshg-sde-differential-form} is understood as a backward SDE.
We emphasise that $\tr(\id) = O(\epsilon^{-2})$, and thus, the bounds below are not uniform in $\epsilon>0$.
By the mean value theorem, for any $\Phi \in X_\epsilon$ there is $\bar \Phi \in X_\epsilon$,
such that
\begin{equation}
\nabla v_s^\ShG (\Phi) = \nabla v_s^\ShG (\Phi) - \nabla v_s^\ShG(0) = \He v_s^\ShG (\bar \Phi) \Phi.
\end{equation}
Using this together with Lemma \ref{lem:lvshg-he-v-pos-definite} for $\cE = \ShG$ the finite variation integrals in \eqref{eq:ito-ftphi-integrated} can be estimated as follows:
\begin{align}
\label{eq:lvshg-sde-existence-finite-variation-integrals-upper}
- & 2  \int_{t\vee T_n}^\infty g_s  \big[ \dot c_s^{-1} \Phi_s^\ShG \dot c_s \nabla v_s^\ShG(\tilde \Phi_s^\ShG) -  \frac{1}{2}\Phi_s^\ShG \ddot c_s \dot c_s^{-2} \Phi_s^\ShG \big ]ds
- \int_{t\vee T_n}^\infty g_s' \|q_s \Phi_s^\ShG\|^2 ds 
\nnb 
&= - 2 \int_{t\vee T_n}^\infty g_s (\dot c_s^{-1} \Phi_s^\ShG) \big[ \dot c_s \He v_s^\ShG (\bar\Phi_s) \dot c_s -\frac{1}{2} \ddot c_s \big]  (\dot c_s^{-1} \Phi_s^\ShG) ds 
- \int_{t\vee T_n}^\infty g_s' \| q_s \Phi_s^\ShG \|^2 ds
\nnb
&\leq \int_{t\vee T_n}^\infty g_s (\dot c_s^{-1} \Phi_s^\ShG)  \ddot c_s (\dot c_s^{-1}\Phi_s^\ShG) \, ds
- \int_{t\vee T_n}^\infty g_s' \| q_s \Phi_s^\ShG\|^2 ds
\nnb
&=  \int_{t\vee T_n}^\infty g_s  q_s \Phi_s^\ShG \ddot c_s c_s^{-1} q_s \Phi_s^\ShG \, ds
- \int_{t\vee T_n}^\infty g_s' \| q_s \Phi_s^\ShG\|^2 ds
\nnb
&\leq  - 2 \int_{t\vee T_n}^\infty ( \frac{m^2}{sm^2+1} + \frac{1}{2}\frac{g_s'}{g_s} ) g_s \| \Phi_s^\ShG \|_{q_s}^2 \, ds 
.
\end{align}
In the last line, we used that
\begin{equation}
\ddot c_s c_s^{-1} = - \frac{2}{s} (-\Delta + m^2) c_s \dot c_s \dot c_s^{-1}
= -2 (-\Delta + m^2) (s (-\Delta + m^2) +\id)^{-1} \leq - 2 \frac{m^2}{s m^2 + 1} \id,
\end{equation}
which holds by the monotonicity of the function $x \mapsto \frac{x}{sx+1}$.
With the choice
\begin{equation}
\label{eq:choice-of-g}
g_s = \frac{1}{(s m^2+1)^2 } \implies \frac{m^2}{sm^2+1} +  \frac{1}{2}\frac{g_t'}{g_t} =0,
\end{equation}
the integral on the right hand side of \eqref{eq:lvshg-sde-existence-finite-variation-integrals-upper} vanishes,
and we obtain from \eqref{eq:ito-ftphi-integrated}
\begin{align}
\label{eq:ftphi-upper-bound}
f(t\vee T_n,\Phi_{t\vee T_n}^\ShG) 
&\leq \int_{t\vee T_n}^\infty \tr(\id) g_s ds
+ 2 \int_{t \vee T_n}^\infty (\dot c_s^{-1} \Phi_s^\ShG) g_s \cdot \dot c_s^{1/2} d W_s
\nnb
& \leq  \tr(\id) \int_t^\infty g_s ds 
+ 2 \int_{t \vee T_n}^\infty (\dot c_s^{-1} \Phi_s^\ShG) g_s \cdot \dot c_s^{1/2} d W_s.
\end{align}
Now, taking expectation, the stochastic integral on the right hand side of \eqref{eq:ftphi-upper-bound} vanishes thanks to the presence of $T_n$ and we obtain
\begin{align}
\label{eq:expectation-ftphi-upper}
\E \big[ f(t\vee T_n,\Phi_{t\vee T_n}^\ShG)  \big] 
&\leq \tr(\id) \int_t^\infty g_s ds
\leq 
\tr(\id) \frac{1}{m^2(t m^2 +1)} .
\end{align}
It follows that
\begin{equation}
\E \big[ \mathbf{1}_{\{T_\infty >0\}} f(t\vee T_n, \Phi_{t\vee T_n}^\ShG) \big] \lesssim \frac{1}{m^2\epsilon^2}
\end{equation}
and thus, taking $t\to 0$, we have by the dominated convergence theorem
\begin{equation}
\E \big[ \mathbf{1}_{\{T_\infty >0\}} f(T_n, \Phi_{T_n}^\ShG) \big] = \P(T_\infty >0) n
 \leq \frac{1}{m^2\epsilon^2} ,
\end{equation}
which, when $n\to \infty$, allows us to conclude $\P(T_\infty >0) = 0$.

To prove pathwise uniqueness,
we let $\Phi_t^\ShG$ and $\tilde \Phi_t^\ShG$ be two solutions to \eqref{eq:lvshg-sde-differential-form} and define
\begin{equation}
S_n = \sup \{ t\geq 0 \colon g_t \|  q_t ( \Phi_t^\ShG - \tilde \Phi_t^\ShG) \|_{L^2}^2 \geq n^{-1}\}.
\end{equation}
Moreover, we define $S_\infty = \lim_{n\to \infty} S_n$, which is well-defined as $S_n$ is increasing.
Note that $(S_n)_n$ and $S_\infty$ are stopping times with respect to the backward filtration $\cF^t$.
We first observe that
\begin{equation}
\label{eq:inclusion-for-pathwise-uniqueness}
\bigcap_{n\in \N} \{S_n =\infty  \} \subseteq \{ \Phi_t^{\ShG} = \tilde \Phi_t^{\ShG} \text{ for all } t\geq 0  \}
,
\end{equation}
and thus, pathwise uniqueness follows, once we showed that 
the event on the left hand side of \eqref{eq:inclusion-for-pathwise-uniqueness} has probability one.
Using similar calculations as the ones leading to \eqref{eq:ftphi-upper-bound}, we obtain
\begin{align}
\label{eq:ftphi-tphi-upper-bound}
f(t\vee S_n,\Phi_{t\vee S_n}^\ShG - \tilde \Phi_{t\vee S_n}^\ShG ) 
\leq 2 \int_{t \vee S_n}^\infty \dot c_s^{-1}( \Phi_s^\ShG -\tilde  \Phi_s^\ShG)g_s \cdot \dot c_s^{1/2} d W_s . 
\end{align}
Taking expectation, the stochastic integral vanishes by a similar reasoning as above \eqref{eq:expectation-ftphi-upper}, and we obtain
\begin{equation}
\E \big[ f(t\vee S_n , \Phi_{t\vee S_n}^\ShG - \tilde \Phi_{t\vee S_n}^\ShG  \big]  = 0.
\end{equation}
It follows that, when $t\to 0$,
\begin{equation}
\P(S_n < \infty) n^{-1} = 0 \qquad \implies \qquad \P(S_n < \infty) = 0.
\end{equation}
Since $(S_n)_n$ is increasing as $n\to \infty$, we conclude that
\begin{equation}
\P \big( \bigcap_{n\in \N}\{ S_n = \infty\} \big) = 1
\end{equation}
as needed.
\end{proof}


\begin{proof}[Proof of Proposition \ref{prop:law-polchinski-sde}]
We use a similar argument as in the proof of Proposition \ref{prop:lvshg-coupling-eps}.
Let $\Phi_t^{\ShG,T}$ be solution to the backward SDE
\begin{equation}
\label{eq:lvshg-backward-SDE-T}
d \Phi_t = - \dot c_t \nabla v_t^\ShG(\Phi_t) dt + \dot c_t^{1/2} dW_t, \qquad t\in [0,T], \qquad \Phi_T \sim \nu_T^\ShG.
\end{equation}
Using the same argument as in the proof of Proposition \ref{prop:lvshg-coupling-eps} it can be shown that a unique solution to \eqref{eq:lvshg-backward-SDE-T}, henceforth denoted $\Phi^{\ShG,T}$, exists.
Then, by similar arguments as in \cite[Proposition 2.1]{MR4303014}, we have that $\Phi_t^{\ShG,T} \sim \nu_t^\ShG$ for all $t\leq T$.
The ergodicity assumption \cite[(2.5)]{MR4303014} can be proved by elementary means in this case,
but is not needed to conclude this, for which we omit its proof here.
In what follows, we show that, for any $t\geq 0$, we have $\Phi_t^{\ShG,T} \to \Phi_t^\ShG$ in probability as $T\to \infty$.
Let $S_n$ be defined by
\begin{equation}
S_n = \sup\{ t\leq T\colon g_t\|q_t (\Phi_t^\ShG - \Phi_t^{\ShG,T})  \|_{L^2}^2 \geq n \} .
\end{equation}
Using the same arguments as in the proof of existence in Proposition \ref{prop:lvshg-coupling-eps},
now with $\Phi_t^\ShG - \Phi_t^{\ShG,T}$ in place of $\Phi_t^{\ShG}$, we have
\begin{align} 
\label{eq:ito-ftphi-phiT-integrated}
f(t\vee S_n, \Phi_{t\vee S_n}^\ShG - \Phi_t^{\ShG,T})
&-
f(T\vee S_n, \Phi_{T\vee S_n}^\ShG - \Phi_{T\vee S_n}^{\ShG,T} )
\nnb
&= - 2  \int_{t\vee S_n}^T  g_s \dot c_s^{-1} \big( \Phi_s^\ShG - \Phi_s^{\ShG,T} ) \big(  \dot c_s \nabla v_s^\ShG(\Phi_s^\ShG) - \dot c_s \nabla v_s^{\ShG}(\Phi_s^{\ShG,T}) \big) ds
\nnb
&\qquad + \int_{t\vee S_n}^T g_s (\Phi_s^\ShG - \Phi_s^{\ShG,T} )\ddot c_s \dot c_s^{-2} (\Phi_s^\ShG - \Phi_s^{\ShG,T} ) ds
\nnb
& \qquad - \int_{t\vee S_n}^T g_s' \| q_s (\Phi_s^\ShG - \Phi_s^{\ShG,T}) \|_{L^2}^2 \, ds
\nnb
&\qquad + 2 \int_{t\vee S_n}^T g_s \big(\dot c_s^{-1} (\Phi_s^\ShG-\Phi_s^{\ShG,T})\big) \cdot \dot c_s^{1/2} d W_s
\nnb
&\leq 2 \int_{t\vee S_n}^T g_s \big(\dot c_s^{-1}( \Phi_s^\ShG  -\Phi_s^{\ShG,T})\big) \cdot \dot c_s^{1/2} d W_s .
\end{align}
Taking the expectation, the stochastic integral vanishes thanks to the stopping time, and we obtain
\begin{equation}
\E\big[ f(t\vee S_n, \Phi_{t\vee S_n}^\ShG - \Phi_{t\vee S_n}^{\ShG,T}) \big]
\leq
\E \big[ f(T\vee S_n, \Phi_{T\vee S_n}^\ShG - \Phi_{T\vee S_n}^{\ShG,T} )\big] .
\end{equation}
Similarly as before, we have that $S_\infty=0$ a.s., and thus, we have, as $n\to \infty$, from the monotone convergence theorem
\begin{equation}
\E\big[ f(t, \Phi_t^\ShG - \Phi_t^{\ShG,T}) \big]
\leq
\E \big[ f(T, \Phi_T^\ShG - \Phi_T^{\ShG,T} )\big] .
\end{equation}
Now, we use the triangle inequality, the Brascamp-Lieb inequality and the observation that
\begin{equation}
\E[\Phi_T^\GFF \dot c_T^{-1} \Phi_T^\GFF] = O(T\epsilon^{-2})
\end{equation}
as well as \eqref{eq:choice-of-g} to conclude that, as $T\to \infty$,
\begin{equation}
\E \big[ f(T, \Phi_T^\ShG - \Phi_T^{\ShG,T} )\big] \to 0 .
\end{equation}
In total, we have that, as $T\to \infty$ and for any $t\geq 0$
\begin{equation}
\E\big[ f(t, \Phi_t^\ShG - \Phi_t^{\ShG,T}) \big] \to 0.
\end{equation}
Since $g_t > 0$ for any $t\geq 0$, it follows that $\Phi_t^{\ShG,T} \to \Phi_t^\ShG$ in probability as $T\to \infty$.

\end{proof}

\section{Regularity of the multiplicative chaos}
\label{app:wickexp-regularity}

In this section we give the proofs of Lemma \ref{lem:sobolev-norms-liouville-finite} and Lemma \ref{lem:sobolev-norms-sinh-finite}.
For $\cE=\Lv$ this result is standard, and can be for instance extracted from \cite{MR4238209}.
Here we give a proof which can be adjusted to the case $\cE= \ShG$ with minor modifications thanks to the Brascamp-Lieb inequality and Proposition \ref{prop:density-tggf-tshg}.

Below, we let $\PhiGFF\sim \nu^{\GFF_\epsilon}$ and $\PhitShG \sim \nu^{\ShG_\epsilon,t}$ for $t\geq 0$,
where the latter probability distribution is as in Proposition \ref{prop:density-tggf-tshg}.
For $x\in \Omega_\epsilon$ we define
\begin{align}
\MGFFeps(x) &= e^{\sqrt{\beta}\PhiGFF_x - \frac{\beta}{2} c_\infty(x,x) },
\\
\MpmShGeps(x) &= e^{\sqrt{\beta}\PhitShG_x - \frac{\beta}{2} c_\infty(x,x) },
\end{align}
where we recall that, as $\epsilon \to 0$,
\begin{equation}
c_\infty^\epsilon(x,x) = \frac{1}{2\pi} \log \frac{1}{\epsilon} + O(1).
\end{equation}

\begin{proof}[Proof of Lemma \ref{lem:sobolev-norms-liouville-finite} and Lemma \ref{lem:sobolev-norms-sinh-finite}]

Let $\chi, \tilde \chi \in C^\infty_c(\R^2,[0,1])$,  such that 
\begin{equation}
\supp \tilde \chi \subseteq B_{4/3}, \qquad \supp \chi \subseteq B_{4/3}\setminus B_{3/8},
\end{equation}
where $B_r = {x\in \R^2\colon |x|\leq r}$ denots the Euclidean ball in $\R^2$ centred at the origin, and
\begin{equation}
\tilde \chi(x) + \sum_{j=0}^\infty \chi(x/2^j) = 1, \qquad x\in \R^2 .
\end{equation}
In what follows, we write
\begin{equation}
\label{eq:partition-unity}
\chi_{-1} = \tilde \chi,  \quad \chi_{j} = \chi(\cdot/2^j), \,\,\, j \geq 0,
\end{equation}
and note that $(\chi_j)_{j\geq -1}$ is a dyadic partition of unity with $\sup \chi_i \cap \sup \chi_j = \emptyset$ if $|i-j|>1$ .
For $\epsilon>0$ define $j_\epsilon = \max \{ j \geq -1 \colon  \supp \chi_j  \subset (-\pi/\epsilon, \pi/\epsilon]^2 \}$.
Note that for $j \geq j_\epsilon$, $\supp \chi_j$ may intersect $\partial \{ [-\pi/\epsilon, \pi/\epsilon]^2 \}$.
To avoid ambiguities with the periodisation of $\chi_j$ onto $\Omega_\epsilon^*$, we modify our dyadic partition of unity in \eqref{eq:partition-unity} as follows:
for $j \in \{ -1, \ldots, j_\epsilon\}$ let $\chi_j^\epsilon \in C^\infty_c(\R^2, [0,1])$ be such that for $k \in \Omega_\epsilon^*$ we have
\begin{equation}
\chi_j^\epsilon (k)
=
\begin{cases}
\chi_j(k), \,\qquad & j < j_\epsilon,
\\ 
1- \sum_{j < j_{\epsilon}} \chi_j^\epsilon(k), \,\qquad & j = j_\epsilon,
\\
0, \,\qquad & j >j_\epsilon.
\end{cases}	
\end{equation}
Then we define for $j\geq -1$ the $j$-th Fourier projector $\Delta_j$ by
\begin{equation}
\Delta_j^\epsilon f = \cF^{-1} (\chi_j^\epsilon \hat f),
\end{equation}
where $\hat f \colon \Omega_\epsilon^* \to \C, \, k\mapsto \hat f(k)$ is the Fourier transform of $f$.
By Parseval's identity and the shift invariance of the multiplicative chaos we then have
\begin{align}
\label{eq:lv-sobolev-wick-exp} 
\E \big[\| \wick{\exp(\sqrt{\beta}Y_\infty^\epsilon)}_\epsilon \|_{H^{-1+\delta}(\Omega_\epsilon)}^2 \big]
& \lesssim
\sum_{j=-1}^{\infty} 2^{-2 j (1-\delta) } \E \big[\|\Delta_j^\epsilon \wick{\exp(\sqrt{\beta}Y_\infty^\epsilon)}_\epsilon \|_{L^2(\Omega_\epsilon)}^2 \big]
\nnb
&=     
\sum_{j=-1}^{\infty}2^{-2j (1-\delta)} \Big(\E\big[\big|\big(\Delta_j^\epsilon\wick{\exp(\sqrt{\beta} Y_\infty^\epsilon)}_\epsilon \big)(0)\big|^{2}\big]\Big)
\nnb
&= 
\sum_{j=-1}^{\infty}2^{-2j (1-\delta)}\E \Big[\int_{\Omega_\epsilon\times\Omega_\epsilon} K_j^\epsilon(x)K_j^\epsilon (y)
e^{\sqrt{\beta}(Y_\infty^\epsilon(x) + Y_\infty^\epsilon (y) ) - \frac{\beta}{2\pi} c_\infty^\epsilon(x,x)}  dx dy \Big]
\nnb
&\lesssim
\sum_{j=-1}^{\infty}2^{-2 j (1-\delta)}
\Big(\int_{\Omega_\epsilon \times\Omega_\epsilon}
K_j^\epsilon (x)K_j^\epsilon (y) |x-y|^{-\beta/2\pi} dx dy \Big).
\end{align}
In the third line, we denoted the real valued kernel of $\Delta_j^\epsilon$ by $K_j^\epsilon$. 
Recall that we have $\|K_j^\epsilon\|_{L^{1}(\Omega_\epsilon)}\lesssim 1$ and $\|K_j^\epsilon\|_{L^{\infty}(\Omega_\epsilon)}\lesssim 2^{2j}$ uniformly in $\epsilon>0$.
Moreover, we have that the function $h\colon \Omega_\epsilon \to \R$, $h(x) = |x|^{- \beta/2\pi}$ is in $L^{1+\rho}(\Omega_\epsilon)$ for any $\rho>0$ such that $\rho\beta/4\pi < 1 - \beta/4\pi$ in the sense that
\begin{equation}
\sup_{\epsilon>0} \| h \|_{L^{1+\rho}(\Omega_\epsilon)} < \infty  .
\end{equation}
Thus, we have by Young's inequality for the integral on the right hand side of \eqref{eq:lv-sobolev-wick-exp} for any $\rho>0$ as above
\begin{equation}
\Big|\int_{\Omega_\epsilon \times\Omega_\epsilon}
K_j(x)K_j(y) |x-y|^{-\beta/4\pi} dx dy \Big|
\leq
\| K_j^\epsilon \|_{L^1} \| K_j^\epsilon \|_{L^{1+1/\rho}} \| h\|_{L^{1+\rho}} .
\end{equation}
By interpolation, we have that $\|K_j^\epsilon\|_{L^{1+1/\rho}}\lesssim 2^{2j(1-\rho/(1+\rho))}$.
We now choose $\rho>0$ such that
\begin{equation}
\delta < \rho/(\rho+1) \qquad \text{and}\qquad \rho \beta/4\pi < 1- \beta/4\pi.
\end{equation}
Note that such a choice is always possible. Indeed, the two conditions are equivalent to
\begin{align}
\frac{1}{1-\delta}-1 < \rho \qquad &\text{and} \qquad \rho < \frac{4\pi}{\beta} -1,
\end{align}
which is equivalent to $\rho/(\rho+1) \in (\delta, 1-\beta/4\pi)$. The last open interval is non-empty and the function $\rho \mapsto \rho/(\rho+1)$ attains every value in $(0,1)$ for $\rho >0$.

We conclude from \eqref{eq:lv-sobolev-wick-exp} that
\begin{equation}
\E \big[\| \wick{\exp(\sqrt{\beta}Y_\infty^\epsilon)}_\epsilon \|_{H^{-1+\delta}(\Omega_\epsilon)}^2 \big]
\lesssim \sum_{j=-1}^\infty 2^{2j(1-\delta)} 2^{2j(1-\rho/(\rho+1))} < \infty,
\end{equation}
thereby completing the proof.

For $\cE = \ShG$, we use the same argument together with the Brascamp-Lieb inequality for exponential moments in \eqref{eq:lv-sobolev-wick-exp}.
\end{proof}

\section*{Acknowledgements}
The author thanks Francesco De Vecchi and Seiichiro Kusuoka for helpful communications on the regularity of Wick ordered exponentials,
and furthermore Francesco De Vecchi for additional comments on a preliminary version of the manuscript.

This work was partially carried out during an academic visit at the University of Vienna.
The author appreciates the hospitality.

\bibliography{exp_l2.bbl}
\end{document}